\documentclass{amsart}
\usepackage{hyperref}
\usepackage{fullpage}
\usepackage{amssymb}
\usepackage{tikz}
\usetikzlibrary{matrix,arrows,decorations.pathmorphing,calc,cd}

\newcommand{\Ab}	{\operatorname{Ab}}
\newcommand{\End}       {\operatorname{End}}
\newcommand{\Ext}       {\operatorname{Ext}}
\newcommand{\Hom}       {\operatorname{Hom}}
\newcommand{\Tor}       {\operatorname{Tor}}
\newcommand{\uHom}	{\underline{\Hom}}

\newcommand{\ann}       {\operatorname{ann}}
\newcommand{\cok}       {\operatorname{cok}}
\newcommand{\cof}       {\operatorname{cof}}
\newcommand{\fib}       {\operatorname{fib}}
\newcommand{\img}       {\operatorname{img}}
\newcommand{\op}        {{\operatorname{op}}}
\newcommand{\el}	{\operatorname{el}}

\newcommand{\CA}        {{\mathcal{A}}}
\newcommand{\CB}        {{\mathcal{B}}}
\newcommand{\CC}        {{\mathcal{C}}}
\newcommand{\CD}        {{\mathcal{D}}}
\newcommand{\CF}        {{\mathcal{F}}}
\newcommand{\CH}        {{\mathcal{H}}}
\newcommand{\CM}        {{\mathcal{M}}}
\newcommand{\CP}        {{\mathcal{P}}} 
\newcommand{\CS}        {{\mathcal{S}}}
\newcommand{\CT}        {{\mathcal{T}}}
\newcommand{\CX}        {{\mathcal{X}}}
\newcommand{\CY}        {{\mathcal{Y}}}

\newcommand{\hC}	{\widehat{C}}
\newcommand{\hL}	{\widehat{L}}
\newcommand{\hP}	{\widehat{P}}
\newcommand{\hCB}	{\widehat{\mathcal{B}}}
\newcommand{\hCM}	{\widehat{\mathcal{M}}}

\newcommand{\F}         {{\mathbb{F}}}          
\newcommand{\N}         {{\mathbb{N}}}
\newcommand{\Z}         {{\mathbb{Z}}}
\newcommand{\Q}         {{\mathbb{Q}}}

\newcommand{\al}        {\alpha}
\newcommand{\bt}        {\beta} 

\newcommand{\dl}        {\delta}
\newcommand{\ep}        {\epsilon}
\newcommand{\zt}        {\zeta}
\newcommand{\tht}       {\theta}
\newcommand{\kp}        {\kappa}

\newcommand{\lm}        {\lambda}
\newcommand{\om}        {\omega}

\newcommand{\Sg}        {\Sigma}

\newcommand{\ip}[1]     {\langle #1\rangle}
\newcommand{\sse}       {\subseteq}
\newcommand{\Smash}     {\wedge}

\newcommand{\bigWedge}  {\bigvee}
\newcommand{\ot}        {\otimes}
\newcommand{\st}        {\;|\;}
\newcommand{\xra}       {\xrightarrow}
\newcommand{\era}       {\twoheadrightarrow}
\newcommand{\mra}       {\rightarrowtail}
\newcommand{\xla}       {\xleftarrow}
\newcommand{\ov}[1]     {\overline{#1}}
\newcommand{\bsm}       {\left[\begin{smallmatrix}}
\newcommand{\esm}       {\end{smallmatrix}\right]}
\newcommand{\colim}  {\operatornamewithlimits{\underset{\longrightarrow}{lim}}}

\renewcommand{\:}{\colon}

\newtheorem{theorem}{Theorem}[section]

\newtheorem{lemma}[theorem]{Lemma}
\newtheorem{proposition}[theorem]{Proposition}
\newtheorem{corollary}[theorem]{Corollary}
\theoremstyle{definition}
\newtheorem{remark}[theorem]{Remark}
\newtheorem{definition}[theorem]{Definition}
\newtheorem{example}[theorem]{Example}
\newtheorem{construction}[theorem]{Construction}

\begin{document}
\title{Arithmetic localisation and completion of spectra}
\author{N.~P.~Strickland}

\maketitle 

\begin{abstract}
 This is an exposition of facts about $p$-local spectra, $p$-complete
 spectra and modules over the $p$-complete sphere spectrum, including
 homological criteria for finiteness.  Most things are well-known to
 the experts, with a couple of potential exceptions explained in the
 introduction. 
\end{abstract}

\section{Introduction}
\label{sec-intro}

Let $\CB$ be Boardman's stable homotopy category of spectra
(see~\cite{ma:ssa}, for example).  Given a set $Q$ of primes, we can
define a localised subcategory $\CB[Q^{-1}]$, as will be recalled in
Section~\ref{sec-local}.  Given a single prime $p$, we have a category
$\CB_{(p)}=\CB[(\{\text{primes}\}\setminus\{p\})^{-1}]$ of $p$-local
spectra, and we can also define a subcategory $\hCB_p$ of $p$-complete
spectra, as will be recalled in Section~\ref{sec-complete}.  This
contains the $p$-complete sphere spectrum $S_p$.  One can also
construct $S_p$ as a strictly commutative ring spectrum and then
consider the homotopy category of $S_p$-modules, which we call
$\CM_p$; this will be recalled in Section~\ref{sec-module}.  (Although
we will not give details in this note, there are a number of
applications that work more smoothly if we use $\CM_p$ as the ambient
category rather than $\CB$ or $\CB_{(p)}$ or $\hCB_p$.  For example,
this applies to various questions related to Freyd's Generating
Hypothesis, or to the Chromatic Splitting Conjecture.)

All of these are connective stable homotopy categories in the sense
of~\cite{hopast:ash}, such that the localising subcategory generated
by the unit of the symmetric monoidal product is the whole category.

\begin{remark}\label{rem-notation}
 We will refer to~\cite{hopast:ash} for a number of things, but we
 will use different notation.  There, coproducts are denoted by
 $\bigWedge_iX_i$, the monoidal product of $X$ and $Y$ is denoted by
 $X\Smash Y$, and the corresponding internal function object is
 denoted by $F(X,Y)$.  Here we will follow more recent conventions and
 use the notations $\bigoplus_iX_i$ and $X\ot Y$ and $\uHom(X,Y)$.
 The unit of the symmetric monoidal product will generically be
 denoted by $S$, but by some other symbol such as $S[Q^{-1}]$ in some
 particular cases.  The suspension functor will be denoted by
 $\Sg$, and we write $S^i$ for $\Sg^i S$.
\end{remark}

There are two standard finiteness conditions to consider in this
context. 
\begin{definition}
 Let $\CC$ be a stable homotopy category.
 \begin{itemize}
  \item[(a)] We say that an object $X\in\CC$ is \emph{compact} (or
   \emph{small}) if the functor $\CC(X,-)$ preserves all coproducts.
  \item[(b)] We say that an object $X\in\CC$ is \emph{strongly
    dualisable} if the natural map $\uHom(X,S)\ot X\to\uHom(X,X)$ is an
   isomorphism.  
 \end{itemize}
\end{definition}

In this note we will develop some general theory of the categories
$\CB[Q^{-1}]$, $\hCB_p$ and $\CM_p$, with emphasis on the above
finiteness conditions and their relationship with various homological
criteria.  Almost all of the results explained here are well-known to
the experts, and many of them are treated in~\cite[Chapters 8 and
9]{ma:ssa}.  Other useful classical references include~\cite{bo:lspe}
and~\cite{boka:hlc}.  Nonetheless, we hope that a consolidated and
updated treatment will be useful.  We will need various facts about
the algebraic theory of abelian groups, and for these we will mostly
refer to~\cite{st:ata}, which is again a consolidated exposition of
facts that are mostly well-known.

One fact proved here that is not quite as standard is that any
strongly dualisable object in $\hCB_p$ or $\CM_p$ is equivalent to the
$p$-completion of a finite spectrum.  We have not seen this in the
literature, and we found that it required more work than expected.  It
appears here as part of Proposition~\ref{prop-finite-complete}.  We
will also prove as Theorem~\ref{thm-brown} that $\CM_p$ has Brown
representability for homology.  The corresponding facts for
$\CB[Q^{-1}]$ and $\hCB_p$ appear in the literature, but for $\CM_p$
we need to adapt the method used in~\cite{ne:tba} rather than quoting
the results proved there.

\section{Global spectra}
\label{sec-global}

\begin{definition}\label{defn-CW}
 Let $X$ be a bounded below spectrum.  By a \emph{CW structure} on $X$ we
 mean a diagram 
 \[ \dotsb X_{-2} \to X_{-1} \to X_0 \to X_1 \to X_2 \to \dotsb \to X
 \]
 such that 
 \begin{itemize}
  \item[(a)] $X_i=0$ for $i\ll 0$.
  \item[(b)] The cofibre $X_i/X_{i-1}$ is (equivalent to) a spectrum
   of the form $(T_i)_+\ot S^i$, for some set $T_i$.
  \item[(c)] The cofibre $X/X_i$ has $\pi_j(X/X_i)=0$ for $j\leq i$.  
 \end{itemize}
 We say that the structure is of \emph{finite type} if each set $T_i$
 is finite.  We say that the structure is \emph{finite} if in addition
 we have $T_i=\emptyset$ for almost all $i$.
\end{definition}

We will often produce CW structures (and other similar structures)
using the following construction.
\begin{construction}\label{cons-flip}
 Suppose we have spectra $X^i$ for all $i\in\Z$ together with
 morphisms $f^i\:X^i\to X^{i+1}$ that are identities for $i\ll 0$.
 Let $W_i$ be the fibre of $f^i$ (so $W_i=0$ for $i\ll 0$).  Let
 $X^{-\infty}$ be the value of $X^i$ for $i\ll 0$, and let 
 $X^\infty$ be the sequential colimit (or telescope) of the sequence.  
 Put $X_i=\fib(X_{-\infty}\to X^{i+1})$.  The octahedral
 axiom provides a diagram as follows, in which the rows and columns
 are cofibre sequences, and the middle square is a homotopy pullback: 
 \begin{center}
  \begin{tikzcd}
   & W_i \arrow[d] \arrow[r,equals] & 
   W_i \arrow[d] \\
   X^{-\infty} \arrow[r] \arrow[d,equals] &
   X^i \arrow[r] \arrow[d] & 
   \Sg X_{i-1} \arrow[r] \arrow[d] &
   \Sg X^{-\infty} \arrow[d,equals] \\
   X^{-\infty} \arrow[r] &
   X^{i+1} \arrow[r] \arrow[d] &
   \Sg X_i \arrow[r] \arrow[d] &
   \Sg X^{-\infty} \\
   & \Sg W_i \arrow[r,equals] &
   \Sg W_i 
  \end{tikzcd}
 \end{center}
 In other words, we have a chain of maps 
 \[ \dotsb \to X_{i-2} \xra{g_{i-2}} X_{i-1} \xra{g_{i-1}} X_i
      \xra{g_i} X_{i+1} \to \dotsb  
 \]
 with $X_i=0$ for $i\ll 0$, and $X_i/X_{i-1}=W_i$.  It follows that
 $X_i$ lies in the thick subcategory generated by $\{W_j\st j\leq
 i\}$.  Now let $X_\infty$ be the homotopy colimit of the spectra
 $X_i$.  We find that this is the fibre of the map $X^{-\infty}\to
 X^{\infty}$, and that it lies in the localising subcategory
 generated by $\{W_i\st i\in\Z\}$.
\end{construction}

\begin{lemma}\label{lem-bb-hurewicz}
 Let $X$ be a bounded below spectrum.
 \begin{itemize}
  \item[(a)] If $H_*(X)=0$ then $X$ is contractible.
  \item[(b)] If $i$ is the smallest integer such that $H_i(X)\neq 0$,
   then the Hurewicz map $\pi_i(X)\to H_i(X)$ is an isomorphism.
 \end{itemize}
\end{lemma}
\begin{proof}
 First suppose that $H_*(X)=0$.  By assumption $X$ is bounded below, so
 there exists $n$ such that $\pi_j(X)=0$ for all $j<n$.  The Hurewicz
 theorem then tells us that $\pi_n(X)\simeq H_n(X)=0$.  This allows us
 to apply the Hurewicz theorem one dimension higher, so
 $\pi_{n+1}(X)\simeq H_{n+1}(X)=0$.  We can continue this inductively
 to see that $\pi_i(X)=0$ for all $X$, so $X$ is contractible.
 The proof for~(b) is essentially the same.
\end{proof}

\begin{construction}\label{cons-CW-step}
 Let $X$ be a spectrum with $\pi_i(X)=0$ for $i<b$ (so
 $H_b(X)\simeq\pi_b(X)$ by the Hurewicz Theorem).  If $\pi_b(X)$ is a
 free abelian group, we choose a family of elements $(u_t)_{t\in T}$
 in $\pi_b(X)$ that form a basis.  If $\pi_b(X)$ is not a free abelian
 group, we choose a family of elements $(u_t)_{t\in T}$ in $\pi_b(X)$
 that generates $\pi_b(X)$ and has cardinality as small as possible.
 The maps $u_t\:S^b\to X$ can be assembled to give a map $u\:W\to X$,
 where $W=\bigoplus_{t\in T}S^b=T_+\ot S^b$, and we write $Y$ for
 the cofibre of $u$.
\end{construction}

\begin{lemma}\label{lem-CW-step}
 With notation as in Construction~\ref{cons-CW-step}:
 \begin{itemize}
  \item[(a)] $\pi_i(Y)=H_i(Y)=0$ for $i\leq b$.
  \item[(b)] $H_i(Y)=H_i(X)$ for $i>b+1$.
  \item[(c)] If $\pi_b(X)$ is free abelian then $H_i(Y)=H_i(X)$ for
   $i\geq b+1$.
  \item[(d)] If $H_{b+1}(X)$ is free abelian then so is $H_{b+1}(Y)$.
  \item[(e)] If the groups $H_b(X)=\pi_b(X)$ and $H_{b+1}(X)$ are both
   finitely generated, then so is $H_{b+1}(Y)$.
 \end{itemize}
\end{lemma}
\begin{proof}\leavevmode
 First, we have exact sequences
 $\pi_i(W)\to\pi_i(X)\to\pi_i(Y)\to\pi_{i-1}(W)$ with
 $\pi_i(W)=\pi_i(X)=0$ for $i<b$ and $\pi_b(W)=\bigoplus_{t\in T}\Z$,
 and the map $\pi_b(W)\to\pi_b(X)$ is surjective by construction.  It
 follows easily that $\pi_i(Y)=0$ for $i\leq b$, and so $H_i(Y)=0$ for
 $i\leq b$ by the Hurewicz Theorem.  This proves~(a).
  
 Next, we have exact sequences
 $H_i(W)\to H_i(X)\to H_i(Y)\to H_{i-1}(W)$.  If $i\not\in\{b,b+1\}$
 then $H_i(W)=H_{i-1}(W)=0$ so $H_i(Y)=H_i(X)$, which proves~(b).  The
 more interesting part of the sequence has the form
 \[ H_{b+1}(X) \mra H_{b+1}(Y) \to H_b(W) \xra{u_*}
     H_b(X) \era H_b(Y).
 \]
 The map $u_*$ is surjective by construction, so $H_b(Y)=0$, as we
 already saw in~(a).  If $H_b(X)$ is free abelian then $u_*$ is an
 isomorphism by construction and so $H_{b+1}(Y)=H_{b+1}(X)$; this
 proves~(c).  In general, let $F$ be the kernel of $u_*$.  This is a
 subgroup of the free abelian group $\pi_b(W)$, so it is again a free
 abelian group.  (This standard fact is proved
 as~\cite[Theorem 6.6]{st:ata}, for example.)  It follows that the
 short exact sequence $H_{b+1}(X)\to H_{b+1}(Y)\to F$ must split,
 giving $H_{b+1}(Y)=H_{b+1}(X)\oplus F$.  Claim~(d) is clear from this.
 Moreover, if $\pi_b(X)$ is finitely generated then the same is true
 of $\pi_b(W)$ and $F$, so if $H_{b+1}(X)$ is also finitely generated
 then $H_{b+1}(Y)$ has the same property.  This proves~(e).
\end{proof}

\begin{proposition}\label{prop-CW}\leavevmode
 Any bounded below spectrum $X$ admits a CW structure
 $\{X_n\}_{n\in\Z}$ (with $X_n/X_{n-1}=(T_n)_+\ot S^n$ say) such
 that the following additional properties are satisfied:
 \begin{itemize}
  \item[(a)] If $\pi_i(X)=0$ for all $i<b$, then $X_i=0$ and
   $T_i=\emptyset$ for all $i<b$. 
  \item[(b)] If $H_t(X)$ is free abelian and $H_i(X)=0$ for all $i>t$
   then $X_i=X$ for $i\geq t$, and $T_i=\emptyset$
   for $i>t$.  (In particular, this works if $H_i(X)=0$ for all
   $i\geq t$.)
  \item[(c)] If $H_i(X)$ is finitely generated for all $i$ then $T_i$
   is finite for all $i$.  
  \item[(d)] If $\bigoplus_iH_i(X)$ is finitely generated, then
   $\coprod_iT_i$ is finite.
 \end{itemize}
\end{proposition}
\begin{proof}
 If $X=0$ then everything is clear (with $X_n=0$ and $T_n=\emptyset$
 for all $n$).  Otherwise, let $b$ be the smallest integer such that
 $\pi_b(X)\neq 0$.  Put $X^k=X$ for $k\leq b$.  Then apply
 Construction~\ref{lem-CW-step} repeatedly to give cofibrations
 $W_k\to X^k\to X^{k+1}$ where $W_k=(T_k)_+\ot S^k$.  These will have
 properties as follows:
 \begin{itemize}
  \item[(a)] $\pi_i(X^{k+1})=H_i(X^{k+1})=0$ for $i\leq k$.
  \item[(b)] $H_i(X^{k+1})=H_i(X^k)$ for $i>k+1$.
  \item[(c)] If $\pi_k(X^k)$ is free abelian then
   $H_i(X^{k+1})=H_i(X^k)$ for $i\geq k+1$.
  \item[(d)] If $H_{k+1}(X^k)$ is free abelian then so is $H_{k+1}(X^{k+1})$.
  \item[(e)] If the groups $H_k(X^k)=\pi_k(X^k)$ and $H_{k+1}(X^k)$ are both
   finitely generated, then so is $H_{k+1}(X^{k+1})$.
 \end{itemize}
 Let $X^\infty$ be the homotopy colimit of the spectra $X^k$.  Then
 $\pi_i(X^\infty)$ is the colimit of the groups $\pi_i(X^k)$, which
 are zero for $k>i$, so $\pi_i(X^\infty)=0$ for all $i$, so
 $X^\infty=0$.  Now put $X_i=\fib(X\to X^{i+1})$ as in
 Construction~\ref{cons-flip}.  The homotopy colimit of the spectra
 $X_i$ is then $\fib(X^{-\infty}\to X^\infty)=\fib(X\to 0)=X$, and
 $X_i/X_{i-1}=W_i$.  It follows that this gives a CW structure.
 Property~(a) is clear by construction.  Now suppose that $H_t(X)$ is
 free abelian and $H_i(X)=0$ for $i>t$.  Using Lemma~\ref{lem-CW-step}
 we see that when $k<t$ we have $H_i(X^k)=H_i(X)$ for all $i\geq t$,
 so $H_t(X^k)$ is free abelian and $H_i(X^k)=0$ for $i>t$.  We then
 find that $H_t(X^t)$ may be different from $H_t(X)$, but it is still
 free abelian by Lemma~\ref{lem-CW-step}(d).  Thus, in the next
 application of Construction~\ref{cons-CW-step} we will get
 $H_*(X^{t+1})=0$, but $X^{t+1}$ is bounded below, so $X^{t+1}=0$.
 This means that $X_t=\fib(X\to X^{t+1})=X$, which implies claim~(b).
 Claim~(c) is clear from the finiteness clauses in
 Lemma~\ref{lem-CW-step}, and claim~(d) follows from~(b) and~(c).
\end{proof}

\begin{proposition}\label{prop-finite-global}
 Let $X$ be a bounded below spectrum.  Then the following are equivalent:
 \begin{itemize}
  \item[(a)] $X$ admits a finite CW structure.
  \item[(b)] $X$ lies in the thick subcategory generated by the sphere
   $S=S^0$.
  \item[(c)] $X$ is compact
  \item[(d)] $X$ is strongly dualisable
  \item[(e)] $\bigoplus_iH_*(X)$ is finitely generated over $\Z$.
 \end{itemize}
\end{proposition}
\begin{proof}
 It is clear that condition~(a) implies~(b).  It is standard that $S$
 is compact in $\CB$, so $\CB$ is monogenic as defined
 in~\cite{hopast:ash}.  We therefore see using~\cite[Theorem
 2.1.3(d)]{hopast:ash} that conditions~(b), (c) and~(d) are all
 equivalent.  (Moreover, this does not require the hypothesis that $X$
 is bounded below.)  Put
 $\CC=\{X\st H_*(X) \text{ is finitely generated }\}$, and note that
 this is preserved by suspension and desuspension.  Suppose we have a
 cofibration $X\xra{f}Y\xra{g}Z\xra{h}\Sg X$ with $X,Z\in\CC$.  We
 then have a long exact sequence $H_*(X)\to H_*(Y)\to H_*(Z)$ with
 $H_*(X)$ and $H_*(Z)$ finitely generated.  This in turn gives a short
 exact sequence $A_*\to H_*(Y)\to B_*$ where $A_*$ is a quotient of
 $H_*(X)$ and $B_*$ is a subgroup of $H_*(Z)$, so both $A_*$ and $B_*$
 are finitely generated.  This implies that $H_*(Y)$ is finitely
 generated, so $Y\in\CC$.  Similarly, if $X$ is a retract of $Y$ and
 $Y\in\CC$ then $H_*(X)$ is a retract of the finitely generated group
 $H_*(Y)$, so it is also finitely generated, so $X\in\CC$.  This
 proves that $\CC$ is a thick subcategory and it clearly contains $S$.
 Using this we see that~(b) implies~(e).  Finally,
 Proposition~\ref{prop-CW} implies that~(e) implies~(a).
\end{proof}

\section{Moore spectra}
\label{sec-moore}

\begin{definition}\label{defn-moore}
 A \emph{Moore spectrum} is a spectrum $M$ with $\pi_kM=0$ for $k<0$
 (so $\pi_0(M)=H_0(M)$ by the Hurewicz theorem) and $H_k(M)=0$ for $k>0$.  
 We say that a Moore spectrum $M$ is \emph{free} if $H_0(M)$ is a free
 abelian group (and similarly for other adjectives that can be applied
 to abelian groups, such as \emph{finitely generated}).  Given an
 abelian group $A$, we write $SA$ for any Moore spectrum $M$ equipped
 with an isomorphism $A\to\pi_0(M)$.
\end{definition}

\begin{example}\label{eg-S-mod-n}
 For any $n>0$ we define $S/n$ to be the cofibre of $n.1_S\:S\to S$;
 this is a Moore spectrum with $\pi_0(S/n)=\Z/n$.
\end{example}

Note that for any spectra $X$ and $Y$ we have a natural map 
\[ \pi_0 \: [X,Y] \to \Hom(\pi_0X,\pi_0Y). \]
If $X$ is a Moore spectrum then this will often be an isomorphism; we
will give various results in this direction below.

\begin{lemma}\label{lem-moore-iso}
 If $f\:X\to Y$ is a map of Moore spectra and
 $\pi_0(f)\:\pi_0(X)\to\pi_0(Y)$ is an isomorphism then $f$ is an
 equivalence. 
\end{lemma}
\begin{proof}
 The assumption can be rephrased as saying that $H_0(f)$ is an
 isomorphism, so $H_*(\cof(f))=0$.  For $i<0$ we also have
 $\pi_i(X)=\pi_i(Y)=0$ and so $\pi_i(\cof(f))=0$.  It follows by the
 Hurewicz Theorem that $\cof(f)=0$, so $f$ is an equivalence.
\end{proof}

\begin{remark}\label{rem-moore-smash}
 If $M$ and $N$ are Moore spectra with $\Tor(\pi_0(M),\pi_0(N))=0$,
 then the K\"unneth theorem shows that $M\ot N$ is a Moore spectrum
 with $\pi_0(M\ot N)=\pi_0(M)\ot\pi_0(N)$.  In particular, this
 works if either $\pi_0(M)$ or $\pi_0(N)$ is torsion free.
\end{remark}

\begin{remark}\label{rem-moore-cons}
 Given an abelian group $A$, we can construct a Moore spectrum $SA$ as
 follows.  We first choose a free abelian group
 $Q_0=\bigoplus_{j\in J}\Z$ and a surjective
 homomorphism $e_0\:Q_0\to A$.  Then $\ker(e_0)$ is a subgroup of a
 free abelian group and so is free abelian, so we can choose a group
 $P_0=\bigoplus_{i\in I}\Z$ and an isomorphism $P_0\simeq\ker(e_0)$,
 giving a short exact sequence $P_0\xra{f_0}Q_0\xra{e_0}A$.  Put
 $P=\bigoplus_{i\in I}S$ and $Q=\bigoplus_{j\in J}S$ so $\pi_0(P)=P_0$
 and $\pi_0(Q)=Q_0$, so we have a natural map 
 \[ \pi_0 \: [P,Q] \to \Hom(\pi_0(P),\pi_0(Q)) = \Hom(P_0,Q_0). \]
 It is easy to see that this is an isomorphism, so there is a unique
 map $f\:P\to Q$ with $\pi_0(f)=f_0$.  If we let $M$ be the cofibre of
 $f$, it is easy to check that $M$ is a Moore spectrum of type $SA$. 
\end{remark}

\begin{lemma}\label{lem-moore-maps}
 With $SA=M$ as constructed above, for any spectrum $X$ we have a
 short exact sequence 
 \[ \Ext(A,\pi_{i+1}X) \mra \pi_i\uHom(SA,X) \era \Hom(A,\pi_iX), \]
 which is natural in $X$.
\end{lemma}
\begin{proof}
 The cofibration $P\xra{f}Q\to SA$ gives a fibration
 \[ \uHom(\Sg Q,X) \to \uHom(\Sg P,X) \to \uHom(SA,X) \to \uHom(Q,X) \to \uHom(P,X). \]
 By applying $\pi_i$ we get an exact sequence of abelian groups.
 We have $\pi_i\uHom(Q,X)=[Q,\Sg^{-i}X]$ and this has a natural map to
 $\Hom(\pi_0Q,\pi_0\Sg^{-i}X)=\Hom(Q_0,\pi_iX)$.  As $Q$ is just a
 coproduct of copies of $S$, this map is an isomorphism.  By
 identifying the other terms in a similar way, we get an exact
 sequence 
 \[ \Hom(Q,\pi_{i+1}X) \xra{f_0^*} \Hom(P,\pi_{i+1}X) \to 
     \pi_i\uHom(SA,X) \to
     \Hom(Q,\pi_iX) \xra{f^*} \Hom(P,\pi_iX).
 \] 
 The cokernel of the first map is (essentially by definition)
 $\Ext(A,\pi_{i+1}X)$, and the kernel of the last map is
 $\Hom(A,\pi_iX)$, so we have an induced short exact sequence as
 claimed. 
\end{proof}

\begin{remark}\label{rem-moore-maps}
 We can apply the above when $X$ is another Moore spectrum of type
 $SA$ constructed in some other way, and we can conclude that
 $X\simeq SA$, so all Moore spectra of type $SA$ are equivalent.  It
 also follows that the map $\pi_0\:[SA,SB]\to\Hom(A,B)$ is surjective,
 with kernel $\Ext(A,\pi_1SB)$.  This is often zero but not always,
 because of the fact that $\pi_1(S)=\Z/2\neq 0$.  Thus, it is almost
 correct but not quite correct to think of $SA$ as being functorial in
 $A$.
\end{remark}

\begin{example}\label{eg-moore-exponent}
 Let $p$ be a prime.  It is a standard fact that $\pi_1(S)=\Z/2$, and
 it follows using the cofibration $S\xra{p}S\to S/p$ that
 $\pi_1(S/p)=0$ except when $p=2$ in which case $\pi_1(S/2)=\Z/2$.
 Lemma~\ref{lem-moore-maps} gives a short exact sequence
 \[ \Ext(\Z/p,\pi_1(S/p)) \mra [S/p,S/p] \era \Hom(\Z/p,\Z/p). \]
 This gives $[S/p,S/p]=\Z/p$ when $p$ is odd, and $|[S/2,S/2]|=4$.  In
 particular we have $p^2.1_{S/p}=0$ in all cases.  (It is known that
 $[S/2,S/2]=\Z/4$, but we will not need that.) 
\end{example}

\begin{remark}\label{rem-moore-ring}
 As a special case of Remark~\ref{rem-moore-cons}, suppose we have a
 commutative ring $P_0$ which is free as an abelian group, together
 with an element $x\in P_0$ which is not a zero divisor.  We can the
 form a spectrum $P$ of the form $\bigoplus_{i\in I}S$ with
 $\pi_0(P)=P_0$.  We then find that $P\ot P$ is a Moore spectrum
 with $\pi_0(P\ot P)=P_0\ot P_0$, so there is a unique map
 $\mu\:P\ot P\to P$ such that $\pi_0(\mu)$ is the multiplication
 map $P_0\ot P_0\to P_0$.  This makes $P$ into a homotopy commutative
 ring spectrum.  There is also a unique map $\mu_x\:P\to P$ with
 $\pi_0(\mu_x)(a)=ax$ for all $a$.  We define $P/x$ to be the cofibre
 of $\mu_x$, and we find that $P/x$ is a Moore spectrum with
 $\pi_0(P/x)=P_0/x$.  Note that $S/n$ can be seen as an example of
 this.   
\end{remark}

\begin{definition}\label{defn-AQi}
 Let $Q$ be a set of prime numbers, and let $\ip{Q}$ be the set of
 positive integers that can be expressed as a product of primes in
 $Q$.   As usual, given an abelian group $A$ we write $A[Q^{-1}]$ for
 the group of fractions $a/m$ with $m\in\ip{Q}$, where
 $a_0/m_0=a_1/m_1$ if and only if there exists 
 $n\in\ip{Q}$ with $n(m_1a_0-m_0a_1)=0$.  Basic facts about this
 construction are reviewed in~\cite[Section 9]{st:ata}, including the
 fact that $A[Q^{-1}]$ is naturally isomorphic to $\Z[Q^{-1}]\ot A$.
\end{definition}

In Section~\ref{sec-local} we will need a Moore spectrum for
$\Z[Q^{-1}]$ (which will be denoted by $S[Q^{-1}]$).  We next give a
convenient construction of such a spectrum, using
Remark~\ref{rem-moore-ring}.

\begin{construction}\label{eg-SQi}
 Let $Q$ be a set of primes.  Any integer $n>0$ can be factored as
 $n_0n_1$, where $n_0$ is a product of primes in $Q$ and $n_1$ is a
 product of primes not in $Q$.  We write $n_Q$ for $n_0$, and call
 this the \emph{$Q$-factor} of $n$.

 Now let $m_r$ be the $Q$-factor of $r!$, and let $m_{r,s}$ be the
 $Q$-factor of $\binom{r+s}{r}$, so $m_0=m_1=1$ and
 $m_{r+s}=m_rm_sm_{r,s}$.  Let $DP_0(Q)$ be the $\Z$-linear span of the
 elements $t_r=t^r/m_r\in\Q[t]$.  These satisfy
 $t_rt_s=m_{r,s}t_{r+s}$, so $DP_0(Q)$ is a subring of $\Q[t]$ which is
 free abelian as an additive group.  There is a ring map
 $\phi\:DP_0(Q)\to\Z[Q^{-1}]$ given by $\phi(t_r)=1/m_r$, and it is easy
 to see that this is surjective.  Now put $x=t_0-t_1=1-t\in DP(R)$.
 This is not a zero divisor, and it satisfies $\phi(x)=0$, so there is
 an induced surjective homomorphism $\ov{\phi}\:DP_0(Q)/x\to\Z[Q^{-1}]$.
 In $DP_0(Q)$ we have $t_r=m_{1,r}t_{r+1}+(1-t)t_r$, and it is an
 exercise to deduce from this that $\ov{\phi}$ is an isomorphism.
 Thus, we can construct a ring spectrum $DP(Q)=\bigoplus_{n\in\N}S$
 with $\pi_0(DP(Q))=DP_0(Q)$, and a map $\mu_x\:DP(Q)\to DP(Q)$
 corresponding to multiplication by $x$, and then the cofibre of
 $\mu_x$ is a Moore spectrum for $\Z[Q^{-1}]$.
\end{construction}

We could use the above construction to make a Moore spectrum for the
ring $\Z[1/p]$.  However, in this special case it is more convenient
to use a slightly different construction which also gives a compatible
Moore spectrum for the quotient group $\Z/p^\infty=\Z[1/p]/\Z$.

\begin{definition}
 We write $S[t]$ for $\bigoplus_{n\geq 0}S$, and let $i_n$ be the
 inclusion of the $n$th summand.  Let $\mu\:S[t]\ot S[t]\to S[t]$
 be the unique map with $\mu\circ(i_n\ot i_m)=i_{n+m}$ for all
 $n,m\in\N$.  This makes $S[t]$ into a homotopy commutative ring
 spectrum with $\pi_0(S[t])=\Z[t]$ as rings.  We have a map
 $\mu_t\:S[t]\to S[t]$ with $\mu_t\circ i_n=i_{n+1}$, which induces
 multiplication by $t$ on $\pi_0$.  We also have a map
 $\dl_t\:S[t]\to S[t]$ with $\dl_t\circ i_0=0$ and
 $\dl_t\circ i_{n+1}=i_n$.  On $\pi_0$, this corresponds to the
 truncated division operation $f(t)\mapsto (f(t)-f(0))/t$.
\end{definition}

\begin{construction}\label{cons-S-pinv}
 Now fix a prime $p$.  We then have an element
 $1-pt\in\Z[t]=\pi_0(S[t])$, which is not a zero divisor, and a
 multiplication map $\mu_{1-pt}=1-p\mu_t\:S[t]\to S[t]$.  It is easy
 to identify $\Z[t]/(1-pt)$ with $\Z[1/p]$, so the spectrum
 $S[1/p]=\cof(\mu_{1-pt})$ is a Moore spectrum for $\Z[1/p]$.  Now
 define $\phi\:\Z[t]\to\Z/p^\infty=\Z[1/p]/\Z$ by
 $\phi(t^r)=p^{-1-r}+\Z$.  It is an exercise to check that this is a
 cokernel for the map $\dl_t-p=\dl_t\circ\mu_{1-pt}\:\Z[t]\to\Z[t]$,
 which is injective.  Thus, the spectrum $S/p^\infty=\cof(\dl_t-p)$ is
 a Moore spectrum for $\Z/p^\infty$.  We will also use the notation
 $S^{-1}/p^\infty$ for $\Sg^{-1}S/p^\infty$.  The octahedral axiom
 relates the (co)fibres of $\dl_t$, $\mu_{1-pt}$ and the composite
 $\dl_t\circ\mu_{1-pt}$, but the fibre of $\dl_t$ is just $S$: this
 gives a cofibration 
 \[ S^{-1}/p^\infty \to S \to S[1/p] \to S/p^\infty. \]
\end{construction}

\section{$Q$-local spectra}
\label{sec-local}

\begin{definition}
 Again let $Q$ be a set of primes, so $\ip{Q}$ be the set of positive
 integers that can be expressed as a product of primes in $Q$.
 \begin{itemize}
  \item[(a)] We say that an abelian group $A$ is $Q$-local if for all
   $q\in Q$, the map $q.1_A\:A\to A$ is an isomorphism.
  \item[(b)] We say $A$ is $Q$-torsion if for
   all $a\in A$ there exists $m\in\ip{Q}$ with $ma=0$
  \item[(c)] We say $A$ is uniformly $Q$-torsion if there exists
   $m\in\ip{Q}$ with $mA=0$ 
  \item[(d)] We say that a spectrum $X$ is $Q$-local if each homotopy group
   $\pi_i(X)$ is $Q$-local.
  \item[(e)] We say that $X$ is $Q$-torsion if each homotopy group
   $\pi_i(X)$ is $Q$-torsion.
  \item[(f)] We say that $X$ is uniformly $Q$-torsion if there exists
   $m\in\ip{Q}$ with $m.1_X=0$.
  \item[(g)] We write $\CB[Q^{-1}]$ for the full subcategory of
   $Q$-local spectra.
 \end{itemize}
\end{definition}

\begin{remark}
 An abelian group $A$ is $\{p\}$-local if and only if $p.1_A$ is an
 isomorphism.  On the other hand, $A$ is said to be $p$-local if it is
 $Q$-local, where $Q$ is the set of primes different from $p$.  This
 terminology is unfortunate, but well-established.
\end{remark}

\begin{remark}\label{rem-local-test}
 Recall that $f\:X\to Y$ is an equivalence iff
 $\pi_*(f)\:\pi_*(X)\to\pi_*(Y)$ is an isomorphism.  It follows that
 $X$ is $Q$-local iff $q.1_X\:X\to X$ is an equivalence for all
 $q\in Q$.  If so, we see that each homology group $H_i(X)$ is
 $Q$-local.  Conversely, if $X$ is bounded below and all homology
 groups are $Q$-local, then we can use Lemma~\ref{lem-bb-hurewicz} to
 see that $X$ is $Q$-local.
\end{remark}

\begin{remark}
 If $Q$ is the set of all primes then $\Z[Q^{-1}]=\Q$.  We use the
 notation $A\Q$ for $A[Q^{-1}]$ and $\CB\Q$ for $\CB[Q^{-1}]$.

 Now suppose instead that we fix a prime $p$ and let $Q$ be the set of
 all other primes.  We then use the notation $A_{(p)}$ for $A[Q^{-1}]$
 and $\CB_{(p)}$ for $\CB[Q^{-1}]$.
\end{remark}

\begin{definition}
 We also use the notation $L_QS$ for $S[Q^{-1}]$, and we define $C_QS$
 to be the fibre of the unit map $S\to L_QS$.  Given an arbitrary
 spectrum $X$ we write $L_QX$ or $X[Q^{-1}]$ for the spectrum
 $L_QS\ot X=S[Q^{-1}]\ot X$, and we write $C_QX$ for
 $C_QS\ot X$.  The fibration $C_QS\xra{\zt}S\xra{\eta}L_QS$ gives
 a natural fibration $C_QX\xra{\zt}X\xra{\eta}L_QX$.
\end{definition}

\begin{proposition}\label{prop-LQ}\leavevmode
 \begin{itemize}
  \item[(a)] $L_QX$ is always $Q$-local, with
   $\pi_*(L_QX)=\pi_*(X)[Q^{-1}]$.   
  \item[(b)] $C_QX$ is always $Q$-torsion.
  \item[(c)] A spectrum $X$ is $Q$-local iff $C_QX=0$ iff the map
   $\eta\:X\to L_QX$ is an equivalence.
  \item[(d)] A spectrum $X$ is $Q$-torsion iff $L_QX=0$ iff the map
   $\zt\:C_QX\to X$ is an equivalence.
 \end{itemize}
\end{proposition}
\begin{proof}
 We will prove~(a) then~(c) then~(d) then~(b).
 \begin{itemize}
  \item[(a)] Consider an element $q\in Q$, so $q.1_{S[Q^{-1}]}$ is an
   equivalence, so $(q.1_{S[Q^{-1}]})\ot 1_X$ is an equivalence.
   As all relevant constructions are additive functors, this is the
   same as $q.1_{S[Q^{-1}]\ot X}=q.1_{L_QX}$.  It follows that
   $L_QX$ is $Q$-local, so each homotopy group $\pi_i(L_QX)$ is
   $Q$-local.  It follows that the map $\eta_*\:\pi_i(X)\to\pi_i(L_QX)$
   extends uniquely to give a map
   $\eta'\:\pi_i(X)[Q^{-1}]=\Z[Q^{-1}]\ot\pi_i(X)\to\pi_i(L_QX)$.  As
   $\Z[Q^{-1}]$ is torsion-free, the functor $\Z[Q^{-1}]\ot(-)$ is
   exact, so both $\Z[Q^{-1}]\ot\pi_*(X)$ and $\pi_*(L_QX)$ are
   homology theories.  The map $\eta'$ is a morphism of homology
   theories that is an isomorphism when $X=S$, so it is always an
   isomorphism, so $\pi_*(L_QX)=\pi_*(X)[Q^{-1}]$.  
  \item[(c)] Given an abelian group $A$, it is
   standard~\cite[Proposition 9.14]{st:ata} that $A$ is
   $Q$-local if and only if the natural map $A\to A[Q^{-1}]$ is an
   isomorphism.  Thus, a spectrum $X$ is $Q$-local if and only if the
   natural map $\pi_*(X)\to\pi_*(X)[Q^{-1}]=\pi_*(L_QX)$ is an
   isomorphism, if and only if the map $X\to L_QX$ is an equivalence,
   if and only if $C_QX=0$.
  \item[(d)] It is standard that an abelian group $A$ is $Q$-torsion
   if and only if $A[Q^{-1}]=0$.  Thus a spectrum $X$ is $Q$-torsion
   if and only if the graded group $\pi_*(X)[Q^{-1}]=\pi_*(L_QX)$ is
   trivial, if and only if the spectrum $L_QX$ is zero, if and only if
   the map $\zt\:C_QX\to X$ is an equivalence.
  \item[(b)] As $L_QX$ is local we have $C_QL_QX=0$ by~(c).  However,
   \[ C_QL_QX=C_QS\ot L_QS\ot X\simeq
         L_QS\ot C_QS\ot X=L_QC_QX,
   \]
   so $L_QC_QX=0$.  Using~(d) we conclude that $C_QX$ is $Q$-torsion.   
 \end{itemize}
\end{proof}

\begin{corollary}\label{cor-bousfield-LQ}\leavevmode
 \begin{itemize}
  \item[(a)] If $X$ is $Q$-torsion and $Y$ is $Q$-local then $[X,Y]=0$
  \item[(b)] Conversely, if $X$ is such that $[X,Y]=0$ for all
   $Q$-local $Y$, then $X$ is $Q$-torsion.
  \item[(c)] Similarly, if $Y$ is such that $[X,Y]=0$ for all
   $Q$-torsion $X$, then $Y$ is $Q$-local.
  \item[(d)] If $X\in\CB$ and $Y\in\CB[Q^{-1}]$ then the map
   $\eta^*\:[L_QX,Y]\to [X,Y]$ is an isomorphism.
  \item[(e)] Thus, $\CB[Q^{-1}]$ is the localisation of $\CB$ with
   respect to $S[Q^{-1}]$ (in the sense of Bousfield~\cite{bo:lspe},
   also explained in~\cite[Section 3]{hopast:ash}).  The
   corresponding localisation functor is $L_Q$, and the corresponding
   acyclisation functor is $C_Q$.
 \end{itemize}
\end{corollary}
\begin{proof}\leavevmode
 \begin{itemize}
  \item[(a)] Suppose we have $f\:X\to Y$ where $X$ is $X$-torsion and
   $Y$ is $Q$-local, so $L_QX=0$ and the map $Y\to L_QY$ is an
   equivalence.  We have a commutative diagram 
   \begin{center}
    \begin{tikzcd}
     X \arrow[r,"\eta"] \arrow[d,"f"'] &
     L_QX = 0 \arrow[d,"L_Qf"] \\
     Y \arrow[r,"\simeq","\eta"'] & L_QY,
    \end{tikzcd}
   \end{center}
   from which it is clear that $f=0$.
  \item[(b)] If $X$ is as described, then in particular the natural
   map $X\to L_QX$ must be zero, so the fibre $C_QX\to X$ is a split
   epimorphism, so $X$ is a retract of the $Q$-torsion spectrum
   $C_QX$, so $X$ is $Q$-torsion.
  \item[(c)] If $Y$ is as described, then in particular the natural
   map $C_QY\to Y$ must be zero, so the cofibre $X\to L_QX$ is a split
   monomorphism, so $X$ is a retract of the $Q$-local spectrum
   $L_QX$, so $X$ is $Q$-local.
  \item[(d)] The cofibration $C_QX\to X\xra{\eta}L_QX$ gives an exact
   sequence $[\Sg C_QX,Y]\to [L_QX,Y]\xra{\eta^*}[X,Y]\to [C_QX,Y]$ in which
   the first and last terms are zero by~(a), so $\eta^*$ is an isomorphism. 
  \item[(e)] The required axioms are precisely what we have proved. 
 \end{itemize}
\end{proof}

\begin{remark}\label{rem-local-homology}
 The Eilenberg-MacLane spectrum $H$ has $\pi_0(H)=\Z$ and $\pi_i(H)=0$
 for $i\neq 0$.  It follows that the spectrum
 $H[Q^{-1}]=H\ot S[Q^{-1}]$ has $\pi_0(H[Q^{-1}])=\Z[Q^{-1}]$ and
 $\pi_i(H[Q^{-1}])=0$ for $i\neq 0$.  In other words, $H[Q^{-1}]$ is
 the Eilenberg-MacLane spectrum for the ring $\Z[Q^{-1}]$.  It follows
 that 
 \[ H_*(X)[Q^{-1}] = \pi_*(H\ot X)[Q^{-1}] = 
     \pi_*(H\ot X\ot S[Q^{-1}]) = \pi_*(H[Q^{-1}]\ot X) =
    H_*(X;\Z[Q^{-1}]).
 \]
\end{remark}

\begin{definition}\label{defn-CW-local}
 Let $X$ be a bounded below $Q$-local spectrum.  By a \emph{$Q$-local
  CW structure} on $X$ we mean a diagram
 \[ \dotsb X_{-2} \to X_{-1} \to X_0 \to X_1 \to X_2 \to \dotsb \to X
 \]
 such that 
 \begin{itemize}
  \item[(a)] $X_i=0$ for $i\ll 0$.
  \item[(b)] The cofibre $X_i/X_{i-1}$ is (equivalent to) a spectrum
   of the form $\Sg^i(T_i)_+\ot S[Q^{-1}]$ for some set $T_i$.
  \item[(c)] The cofibre $X/X_i$ has $\pi_j(X/X_i)=0$ for $j\leq i$.  
 \end{itemize}
 We say that the structure is of \emph{finite type} if each set $T_i$
 is finite.  We say that the structure is \emph{finite} if in addition
 we have $T_i=\emptyset$ for almost all $i$.
\end{definition}

\begin{construction}\label{cons-CW-step-local}
 Let $X$ be a spectrum such that the group $\pi_i(X)[Q^{-1}]=0$ for
 $i<b$.  This means that the spectrum $X[Q^{-1}]$ has
 $\pi_i(X[Q^{-1}])=0$ for $i<b$, so we can apply the Hurewicz theorem
 to see that $H_i(X[Q^{-1}])=0$ for $i<b$ and
 $H_b(X[Q^{-1}])=\pi_b(X[Q^{-1}])$, or equivalently
 $H_b(X)[Q^{-1}]=\pi_b(X)[Q^{-1}]$.  If $\pi_b(X)[Q^{-1}]$ is a free
 module over $\Z[Q^{-1}]$, we choose a family of elements
 $(u'_t)_{t\in T}$ in $\pi_b(X)[Q^{-1}]$ that form a basis.  If
 $\pi_b(X)[Q^{-1}]$ is not a free module, we choose a family of
 elements $(u'_t)_{t\in T}$ in $\pi_b(X)[Q^{-1}]$ that generates
 $\pi_b(X)[Q^{-1}]$ and has cardinality as small as possible.  Each
 element $u'_t$ can be expressed as $u_t/m_t$ for some element
 $u_t\in\pi_b(X)$ and $m_t\in\ip{Q}$, and the images of these elements
 $u_t$ will again form a basis or generating set.

 The maps $u_t\:S^b\to X$ can be assembled to give a map $u\:W\to X$,
 where $W=\bigoplus_{t\in T}S^b=T_+\ot S^b$, and we write $Y$ for
 the cofibre of $u$.
\end{construction}

\begin{lemma}\label{lem-CW-step-local}
 With notation as in Construction~\ref{cons-CW-step-local}:
 \begin{itemize}
  \item[(a)] For $i\leq b$ we have 
   $\pi_i(Y)[Q^{-1}]=H_i(Y)[Q^{-1}]=0$.
  \item[(b)] For $i>b+1$ we have $H_i(Y)[Q^{-1}]=H_i(X)[Q^{-1}]$.
  \item[(c)] If $\pi_b(X)[Q^{-1}]$ is a free module over $\Z[Q^{-1}]$
   then $H_i(Y)[Q^{-1}]=H_i(X)[Q^{-1}]$ for $i\geq b+1$.
  \item[(d)] If $H_{b+1}(X)[Q^{-1}]$ is a free module over
   $\Z[Q^{-1}]$ then so is $H_{b+1}(Y)[Q^{-1}]$ .
  \item[(e)] If the modules $H_b(X)[Q^{-1}]=\pi_b(X)[Q^{-1}]$ and
   $H_{b+1}(X)[Q^{-1}]$ are both finitely generated over $\Z[Q^{-1}]$,
   then so is $H_{b+1}(Y)[Q^{-1}]$. 
 \end{itemize}
\end{lemma}
\begin{proof}\leavevmode
 The cofibration $W\to X\to Y$ gives a cofibration
 $W[Q^{-1}]\to X[Q^{-1}]\to Y[Q^{-1}]$, which gives exact sequences
 $\pi_i(W)[Q^{-1}]\to\pi_i(X)[Q^{-1}]\to\pi_i(Y)[Q^{-1}]\to\pi_{i-1}(W)[Q^{-1}]$ 
 and
 $H_i(W)[Q^{-1}]\to H_i(X)[Q^{-1}]\to H_i(Y)[Q^{-1}]\to H_{i-1}(W)[Q^{-1}]$.
 These can be used to prove the stated claims, in the same way as in
 Lemma~\ref{lem-CW-step}.  (We need to use the fact that any submodule
 of a free $\Z[Q^{-1}]$-module is again free, which holds because
 $\Z[Q^{-1}]$ is a principal ideal domain.)
\end{proof}

\begin{proposition}\label{prop-CW-local}\leavevmode
 Let $X$ be a bounded below $Q$-local spectrum.  Then we can construct 
 \begin{itemize}
  \item A spectrum $X'$ and a map $f\:X'\to X$ inducing an equivalence 
   $X'[Q^{-1}]\to X$; and
  \item A CW structure $\{X'_n\}_{n\in\Z}$
   (with $X'_n/X'_{n-1}=(T_n)_+\ot S^n$ say)
 \end{itemize}
 such that:
 \begin{itemize}
  \item[(a)] If $\pi_i(X)[Q^{-1}]=0$ for all $i<b$, then $X'_i=0$ and
   $T_i=\emptyset$ for all $i<b$. 
  \item[(b)] If $H_t(X)[Q^{-1}]$ is a free $\Z[Q^{-1}]$-module and
   $H_i(X)[Q^{-1}]=0$ for all $i>t$ then $X'_i=X'$ for $i\geq t$, and
   $T_i=\emptyset$ for $i>t$.  (In particular, this works if
   $H_i(X)[Q^{-1}]=0$ for all $i\geq t$.)
  \item[(c)] If $H_i(X)[Q^{-1}]$ is finitely generated for all $i$
   then $T_i$ is finite for all $i$.  
  \item[(d)] If $\bigoplus_iH_i(X)[Q^{-1}]$ is finitely generated, then
   $\coprod_iT_i$ is finite.
 \end{itemize}
\end{proposition}
\begin{proof}
 If $X[Q^{-1}]=0$ then everything is clear (with $X'=0$, and $X'_n=0$
 and $T_n=\emptyset$ for all $n$).  Otherwise, let $b$ be the smallest
 integer such that $\pi_b(X)[Q^{-1}]\neq 0$.  Put $X^k=X$ for
 $k\leq b$.  Then apply Construction~\ref{lem-CW-step-local}
 repeatedly to give cofibrations $W_k\to X^k\to X^{k+1}$ where
 $W_k=(T_k)_+\ot S^k$.  Let $X^\infty$ be the homotopy colimit of
 the spectra $X^k$.  Then $\pi_i(X^\infty)[Q^{-1}]$ 
 is the colimit of the groups $\pi_i(X^k)[Q^{-1}]$, which are zero for $k>i$,
 so $\pi_i(X^\infty)[Q^{-1}]=0$ for all $i$, so $X^\infty[Q^{-1}]=0$.
 Now put $X'_i=\fib(X\to X^{i+1})$ as in Construction~\ref{cons-flip},
 and $X'=\fib(X\to X^\infty)$, which is the homotopy colimit of the
 spectra $X'_i$.  We then have $X'_i/X'_{i-1}=W_i$, so the spectra
 $X'_i$ give a CW structure on $X'$.  We have a cofibration
 $X'\xra{f}X\to X^\infty$ giving a cofibration
 $X'[Q^{-1}]\to X[Q^{-1}]=X\to X^\infty[Q^{-1}]=0$, showing that
 $X'[Q^{-1}]=X$.  Properties~(a) to~(d) are proved in the same way as
 Proposition~\ref{prop-CW}. 
\end{proof}

\begin{corollary}\label{cor-CW-local}\leavevmode
 Any bounded below $Q$-local spectrum $X$ admits a $Q$-local CW structure
 $\{X_n\}_{n\in\Z}$ (with $X_n/X_{n-1}=\Sg^n (T_n)_+\ot S[Q^{-1}]$ say) such
 that the following additional properties are satisfied:
 \begin{itemize}
  \item[(a)] If $\pi_i(X)=0$ for all $i<b$, then $X_i=0$ and
   $T_i=\emptyset$ for all $i<b$. 
  \item[(b)] If $H_t(X)$ is a free module over $\Z[Q^{-1}]$ and
   $H_i(X)=0$ for all $i>t$ then $X_i=X$ for $i\geq t$, and
   $T_i=\emptyset$ for $i>t$.  (In particular, this works if
   $H_i(X)=0$ for all $i\geq t$.)
  \item[(c)] If $H_i(X)$ is finitely generated over $\Z[Q^{-1}]$ for
   all $i$ then $T_i$ is finite for all $i$.  
  \item[(d)] If $\bigoplus_iH_i(X)$ is finitely generated over
   $\Z[Q^{-1}]$, then $\coprod_iT_i$ is finite.
 \end{itemize}
\end{corollary}
\begin{proof}
 Construct $X'_n$ as in Proposition~\ref{prop-CW-local} and put
 $X_n=X'_n[Q^{-1}]$. 
\end{proof}

\begin{proposition}\label{prop-finite-local}
 Let $X$ be a bounded below $Q$-local spectrum.  Then the following
 are equivalent: 
 \begin{itemize}
  \item[(a)] $X$ admits a finite $Q$-local CW structure.
  \item[(b)] $X$ lies in the thick subcategory generated by 
   $S[Q^{-1}]$.
  \item[(c)] $X$ is compact in $\CB[Q^{-1}]$
  \item[(d)] $X$ is strongly dualisable in $\CB[Q^{-1}]$
  \item[(e)] $\bigoplus_iH_*(X)$ is finitely generated over $\Z[Q^{-1}]$.
  \item[(f)] There is a finite spectrum $X'$ with $L_QX'\simeq X$.
 \end{itemize}
\end{proposition}
\begin{proof}
 It is clear that condition~(a) implies~(b).  Given any family of
 objects $X_i\in\CB[Q^{-1}]$ it is easy to see that
 $\bigoplus_iX_i\in\CB[Q^{-1}]$ and this is the coproduct in
 $\CB[Q^{-1}]$.  Using Corollary~\ref{cor-bousfield-LQ}(d) we see that
 \[ [S[Q^{-1}],\bigoplus_iX_i]=[S,\bigoplus_iX_i] =
     \bigoplus_i[S,X_i] = \bigoplus_i[S[Q^{-1}],X_i],
 \]
 so $S[Q^{-1}]$ is compact in $\CB[Q^{-1}]$.  This means that
 $\CB[Q^{-1}]$ is monogenic as defined
 in~\cite{hopast:ash}.  We therefore see using~\cite[Theorem
 2.1.3(d)]{hopast:ash} that conditions~(b), (c) and~(d) are all
 equivalent.  (Moreover, this does not require the hypothesis that $X$
 is bounded below.)  Just as in Proposition~\ref{prop-finite-global}
 we see that the category of those $X$ satisfying~(e) is thick and
 contains $S[Q^{-1}]$, so~(b) implies~(e).
 Proposition~\ref{prop-CW-local}(d) shows that~(e) implies~(f).
 If~(f) holds, we can apply $L_Q$ to a CW structure on $X'$ to
 deduce~(a). 
\end{proof}

\begin{remark}\label{rem-modules}
 There are various different frameworks in which one can construct a
 more geometric (rather than homotopical) model for $S[Q^{-1}]$ as a
 strictly commutative ring spectrum, and then consider the homotopy
 category of strict $S[Q^{-1}]$-modules, which we will call
 $\CM[Q^{-1}]$.  For example, one can
 use~\cite[Chapter III and Theorem VIII.2.1]{ekmm:rma} or 
 \cite[Section 4.5 and Theorem 7.5.0.6]{lu:ha}.  We will see
 that this does not give anything new in the present case, but we will
 spell out some details for comparison with the $p$-complete case to
 be considered later.

 The properties that we expect from such a framework are encapsulated
 by the definition below. 
\end{remark}

\begin{definition}
 A \emph{strict module category} for $S[Q^{-1}]$ consists of a
 monogenic stable homotopy category $\CM[Q^{-1}]$ together with an
 exact, strongly symmetric monoidal functor $P\:\CB\to\CM[Q^{-1}]$ and
 a right adjoint $U\:\CM[Q^{-1}]\to\CB$ such that $UPX$ is naturally
 isomorphic to $S[Q^{-1}]\ot X$.
\end{definition}

For the rest of this section we assume that we have a strict module
category as above.

\begin{lemma}\label{lem-conservative-local}
 If $M\in\CM[Q^{-1}]$ with $UM=0$, then $M=0$.  Thus, if $f\:M\to N$
 in $\CM[Q^{-1}]$ and $U(f)$ is an isomorphism then $f$ is an
 isomorphism.
\end{lemma}
\begin{proof}
 we have $\CM[Q^{-1}](\Sg^iPS,M)=\CB(S^i,UM)=0$ for all $i$.  Also,
 $P$ is a strongly symmetric monoidal functor, so $PS$ is the unit
 object of $\CM[Q^{-1}]$, and $\CM[Q^{-1}]$ is assumed to be
 monogenic, so from $\CM[Q^{-1}](\Sg^*PS,M)=0$ we can deduce that
 $M=0$ as claimed.  The second claim follows by applying the first to
 $\cof(f)$. 
\end{proof}

\begin{lemma}\label{lem-adjoint-equiv}
 Suppose we have a functor $F\:\CC\to\CD$ with a right adjoint
 $G\:\CD\to\CC$ such that the unit map $\eta\:X\to G(F(X))$ is an
 isomorphism for all $X\in\CC$.  Suppose also that $G$ reflects
 isomorphisms (so if $g\in\CD(Y_0,Y_1)$ and $U(g)$ is an isomorphism
 then $g$ is an isomorphism).  Then the counit maps
 $\ep_Y\:F(G(Y))\to Y$ are also isomorphisms, so $F$ and $G$ give an
 adjoint equivalence between $\CC$ and $\CD$.
\end{lemma}
\begin{proof}
 By the axioms for an adjunction we have a triangular identity
 $G(\ep_Y)\circ\eta_{G(Y)}=1\:G(Y)\to G(Y)$.  As $\eta_{G(Y)}$ is an
 isomorphism we conclude that $G(\ep_Y)$ is an isomorphism.  As $G$
 reflects isomorphisms we conclude that $\ep_Y$ is an isomorphism.
\end{proof}

\begin{proposition}\label{prop-modules-local}
 The functor $U$ takes values in $\CB[Q^{-1}]$ and the restricted
 functors $\CB[Q^{-1}]\xra{P_1}\CM[Q^{-1}]\xra{U_1}\CB[Q^{-1}]$ give
 an adjoint equivalence.
\end{proposition}
This could be seen as an instance
of~\cite[Theorem VIII.3.3]{ekmm:rma}, but we will prove it directly.
\begin{proof}
 We have
 $\CM[Q^{-1}](P(S),P(S))=\CB(S,U(P(S)))=\CB(S,S[Q^{-1}])=\Z[Q^{-1}]$.
 Thus, for every $q\in Q$, multiplication by $q$ is an isomorphism on
 $P(S)$.  For any $M\in\CM[Q^{-1}]$ we have
 $\pi_i(U(M))=\CB(S^i,U(M))=\CM[Q^{-1}](\Sg^iP(S),M)$, and it follows
 that multiplication by $q$ is also an isomorphism on this group.  Thus,
 the functor $U$ takes values in $\CB[Q^{-1}]$.  For
 $X\in\CB[Q^{-1}]$ we have
 $U(P(X))=S[Q^{-1}]\ot X=X[Q^{-1}]\simeq X$, so we see that the unit
 $\eta_X\:X\to U(P(X))$ is an isomorphism.  As $U$ reflects
 isomorphisms (by Lemma~\ref{lem-conservative-local}), we can apply
 Lemma~\ref{lem-adjoint-equiv} to complete the proof.
\end{proof}

\section{$p$-complete spectra}
\label{sec-complete}

In this section we fix a prime number $p$.  We will need various ideas
about $p$-completion and derived $p$-completion of abelian groups,
for which we will follow the treatment in~\cite[Section 12]{st:ata}.
The basic definitions are as follows.
\begin{definition}\leavevmode
 \begin{itemize}
  \item[(a)] We say that an abelian group $A$ is
   \emph{$\Ext$-$p$-complete} if $\Hom(\Z[1/p],A)=\Ext(\Z[1/p],A)=0$. 
  \item[(b)] We define $L_0(A)=\Ext(\Z/p^\infty,A)$ and
   $L_1(A)=\Hom(\Z/p^\infty,A)$. 
 \end{itemize}
\end{definition}

\begin{lemma}
 For any abelian group $A$, there is a natural exact sequence 
 \[ L_1(A) \mra \Hom(\Z[1/p],A) \to A \to L_0(A) \era
     \Ext(\Z[1/p],A). 
 \]
 Thus, there is a natural map $\xi_A\:A\to L_0(A)$, and $A$ is
 $\Ext$-$p$-complete if and only $L_1(A)=0$ and $\xi_A$ is an
 isomorphism. 
\end{lemma}
\begin{proof}
 The sequence is the standard six term exact sequence of $\Hom$ and
 $\Ext$ groups derived from the short exact sequence
 $\Z\to\Z[1/p]\to\Z/p^\infty$, bearing in mind that $\Hom(\Z,A)=A$ and
 $\Ext(\Z,A)=0$.  The second claim is immediate from the exactness of
 the sequence. 
\end{proof}

\begin{remark}
 The groups $L_0(A)=\Ext(\Z/p^\infty,A)$ and
 $L_1(A)=\Hom(\Z/p^\infty,A)$ are naturally modules over the ring
 $\End(\Z/p^\infty)$, which can be identified with the ring $\Z_p$ of
 $p$-adic integers.  If $A$ is $\Ext$-$p$-complete then the map
 $\eta\:A\to L_0A$ is an isomorphism, so $A$ inherits a module
 structure over $\Z_p$.  Naturality of $\eta$ ensures that every
 abelian group homomorphism between $\Ext$-$p$-complete groups is
 actually a morphism of $\Z_p$-modules.  Similarly, for any short
 exact sequence $A\to B\to C$ and any abelian group $T$ we have a
 connecting map $\dl\:\Hom(T,C)\to\Ext(T,A)$ which is natural in $T$
 and thus compatible with endomorphisms of $T$.  By taking
 $T=\Z/p^\infty$ we see that the connecting map $L_1C\to L_0A$ is also
 $\Z_p$-linear.
\end{remark}

\begin{remark}\label{rem-coprod}
 The category $\CA_p$ of $\Ext$-$p$-complete groups is not closed
 under infinite direct sums.  Nonetheless, any family of groups
 $A_i\in\CA_p$ has a coproduct in $\CA_p$, namely the group
 $L_0\left(\bigoplus_iA_i\right)$.  As an example, put
 $P=\bigoplus_{n\in\N}\Z_p$, the group of sequences $(a_n)$ in
 $\prod_{n\in\N}\Z_p$ such that $a_n=0$ for $n\gg 0$.  Let $Q$ be the
 larger group of sequences $a$ that converge to zero $p$-adically, so
 for all $k$ there exists $n$ with $a_i=0\pmod{p^k}$ when $i\geq n$.
 It is not hard to identify $L_0(P)$ with $Q$.
\end{remark}

\begin{lemma}\label{lem-detect-epi}
 If $\al\:A\to B$ is a morphism of $\Ext$-$p$-complete groups and the
 induced map $A/p\to B/p$ is surjective, then $\al$ is surjective.
\end{lemma}
\begin{proof}
 Let $C$ be the cokernel of $\al$.  The functor $A\mapsto A/p$
 is right exact, so $C/p=0$, so $L_0(C)=0$
 by~\cite[Proposition 12.29]{st:ata}.  However, we also know 
 from~\cite[Proposition 12.22(c)]{st:ata} that $C$ is
 $\Ext$-$p$-complete, so $C=L_0(C)=0$, so $\al$ is surjective.
\end{proof}

\begin{lemma}\label{lem-uniform-torsion}
 If $p^eA=0$ for some $e\in\N$, then $A$ is $\Ext$-$p$-complete, so
 $L_1(A)=0$ and $L_0(A)=A$.
\end{lemma}
\begin{proof}
 Multiplication by $p^e$ is an isomorphism on $\Z[1/p]$, so it is also
 an isomorphism on $\Hom(\Z[1/p],A)$ and $\Ext(\Z[1/p],A)$.  On the
 other hand, multiplication by $p^e$ is  zero on $A$, so it is also
 zero on $\Hom(\Z[1/p],A)$ and $\Ext(\Z[1/p],A)$.  These facts can
 only be compatible if $\Hom(\Z[1/p],A)=\Ext(\Z[1/p],A)=0$.
\end{proof}

\begin{definition}
 We write $A[p^e]$ for $\ann(p^e,A)=\{a\in A\st p^ea=0\}$, and
 $A[p^\infty]$ for $\bigcup_eA[p^e]$. 
\end{definition}

\begin{lemma}\label{lem-A-pe}\leavevmode
 \begin{itemize}
  \item[(a)] For any $e\in\N$, the map $\eta\:A\to L_0A$ induces an
   isomorphism $A/p^e\to L_0(A)/p^e$
  \item[(b)] If $L_1A=0$ then for all $e\in\N\cup\{\infty\}$, the map
   $\eta\:A\to L_0(A)$ induces an isomorphism $A[p^e]\to(L_0A)[p^e]$.
 \end{itemize}
\end{lemma}
\begin{proof}
 The case $e=\infty$ in~(b) follows from the other cases, so we can
 assume that $e\in\N$.  We have maps
 \begin{center}
  \begin{tikzcd}
   A[p^e] \arrow[r,"i",rightarrowtail] & 
   A \arrow[r,"f",twoheadrightarrow] & 
   p^eA \arrow[r,"j",rightarrowtail] &
   A \arrow[r,"q",twoheadrightarrow] & 
   A/p^e,
  \end{tikzcd}
 \end{center}
 where $f(a)=p^ea$, and $i$ and $j$ are inclusions, and $q$ is the
 projection, so $jf=p^e.1_A$.  The above diagram breaks into two short
 exact sequences, each of which gives a six-term exact sequence of
 $L_0$ and $L_1$ groups.  Because $p^e$ annihilates $A[p^e]$ and
 $A/p^e$, Lemma~\ref{lem-uniform-torsion} tells us that
 $L_0(A[p^e])=A[p^e]$ and $L_0(A/p^e)=A/p^e$ and
 $L_1(A[p^e])=L_1(A/p^e)=0$.  Thus, the two six-term sequences are
 like
 \begin{center}
  \begin{tikzcd}
   0 \arrow[r,rightarrowtail] & 
   L_1(A) \arrow[r,"L_1(f)",rightarrowtail] & 
   L_1(p^eA) \arrow[r,"\dl"] &
   A[p^e] \arrow[r,"L_0(i)"] & 
   L_0(A) \arrow[r,"L_0(f)",twoheadrightarrow] & 
   L_0(p^eA) 
  \end{tikzcd}
  \begin{tikzcd}
   L_1(p^eA) \arrow[r,"L_1(j)","\simeq"'] & 
   L_1(A) \arrow[r,"L_1(q)"] & 
   0 \arrow[r,"\dl"] &
   L_0(p^eA) \arrow[r,"L_0(j)",rightarrowtail] & 
   L_0(A) \arrow[r,"L_0(f)",twoheadrightarrow] & 
   A/p^e 
  \end{tikzcd}
 \end{center}
 As $L_0(f)$ is surjective we have 
 \[ L_0(A)/p^e=\cok(p^e.1_{L_0(A)})=
    \cok(L_0(j)L_0(f))=\cok(L_0(j))=A/p^e, 
 \]
 which proves~(a).  If $L_1(A)=0$ then the second sequence gives
 $L_1(p^eA)=0$, so the first sequence gives $A[p^e]=\ker(L_0(f))$.
 The second sequence also shows that $L_0(j)$ is injective so 
 \[ \ker(L_0(f))=\ker(L_0(jf))=\ker(p^e.1_{L_0(A)}) = 
      L_0(A)[p^e],
 \]
 which proves~(b).
\end{proof}

\begin{definition}\label{defn-p-complete}\leavevmode
 \begin{itemize}
  \item[(a)] We say that a spectrum $X$ is \emph{$p$-complete} if each
   homotopy group $\pi_i(X)$ is $\Ext$-$p$-complete.
  \item[(b)] We say that $X$ is \emph{$\{p\}$-local} if each homotopy group
   $\pi_i(X)$ is $\{p\}$-local, or equivalently, $p.1_X$ is an
   equivalence, or equivalently, the spectrum
   $X/p=S/p\ot X=\cof(p.1_X)$ is zero.
  \item[(c)] We say that $X$ is \emph{uniformly $p$-torsion} if
   $p^n.1_X=0$ for some $n$.
  \item[(d)] We write $\hCB_p$ for the full subcategory of $p$-complete
   spectra.
 \end{itemize}
\end{definition}

\begin{lemma}
 If $X$ is uniformly $p$-torsion, then it is $p$-complete.
\end{lemma}
\begin{proof}
 This follows from Lemma~\ref{lem-uniform-torsion}.
\end{proof}

\begin{definition}
 We write $\hL_pX$ or $X_p$ for the spectrum $\uHom(S^{-1}/p^\infty,X)$,
 and $\hC_pX$ for $\uHom(S[1/p],X)$.  The cofibration sequence
 $S^{-1}/p^\infty\to S\to S[1/p]$ gives a natural fibration sequence 
 $\hC_pX\xra{\xi}X\xra{\tht}\hL_pX$.
\end{definition}

\begin{lemma}\label{lem-pi-hL}
 For any spectrum $X$, there are natural short exact sequences 
 \begin{align*}
  \Ext(\Z[1/p],\pi_{i+1}(X)) &\mra
   \pi_i(\hC_pX) \era \Hom(\Z[1/p],\pi_i(X)) \\
  \Ext(\Z/p^\infty,\pi_i(X)) &\mra 
   \pi_i(\hL_pX) \era \Hom(\Z/p^\infty,\pi_{i-1}(X)) \\
  L_0(\pi_i(X)) &\mra 
   \pi_i(\hL_pX) \era L_1(\pi_{i-1}(X))
 \end{align*}
\end{lemma}
\begin{proof}
 The first two sequences are special cases of
 Lemma~\ref{lem-moore-maps}, and the third is the same as the second
 but with different notation.
\end{proof}

\begin{proposition}\label{prop-hL}\leavevmode
 \begin{itemize}
  \item[(a)] $\hL_pX$ is always $p$-complete.   
  \item[(b)] $\hC_pX$ is always $\{p\}$-local.
  \item[(c)] A spectrum $X$ is $p$-complete iff $\hC_pX=0$ iff the map
   $\tht\:X\to\hL_pX$ is an equivalence.
  \item[(d)] A spectrum $X$ is $\{p\}$-local iff $\hL_pX=0$ iff the map
   $\xi\:\hC_pX\to X$ is an equivalence.
 \end{itemize}
\end{proposition}
\begin{proof}\leavevmode
 \begin{itemize}
  \item[(a)] Lemma~\ref{lem-pi-hL} gives a short exact sequence
   $L_0(\pi_i(X))\to\pi_i(\hL_pX)\to L_1(\pi_{i-1}X)$.  The first and
   last terms are $\Ext$-$p$-complete
   by~\cite[Proposition 12.28]{st:ata}, so $\pi_i(\hL_pX)$ is
   $\Ext$-$p$-complete by~\cite[Proposition 12.22(a)]{st:ata}.
  \item[(b)] Because $p$ times the identity is an equivalence on the
   spectrum $S[1/p]$, the same is true after applying the additive functor
   $\uHom(-,X)$, so the spectrum $\hC_pX=\uHom(S[1/p],X)$ is $\{p\}$-local.
  \item[(c)] If $X$ is $p$-complete then $L_1(\pi_{i-1}X)=0$ and the
   map $\pi_i(X)\to L_0(\pi_i(X))$ is an isomorphism.  Using the last
   short exact sequence from Lemma~\ref{lem-pi-hL} again, we see that
   the map $\tht\:X\to\hL_pX$ is an equivalence, so $\hC_pX=0$.
   Conversely, we know from~(a) that $\hL_pX$ is always $p$-complete,
   so if $\tht\:X\to\hL_pX$ is an equivalence, then $X$ must also be
   $p$-complete.
  \item[(d)] Suppose that $X$ is $\{p\}$-local.  By
   Corollary~\ref{cor-bousfield-LQ}(a), for any $\{p\}$-torsion
   spectrum $T$, we have $[T,X]=0$.  Taking $T=\Sg^{i-1}S/p^\infty$,
   we get $\pi_i(\hL_pX)=0$.  As this holds for all $i$, we have
   $\hL_pX=0$.  Conversely, if $\hL_pX=0$ then $X\simeq\hC_pX$ and
   $\hC_pX$ is $\{p\}$-local by~(b) so $X$ is $\{p\}$-local.
 \end{itemize}
\end{proof}

\begin{corollary}\label{cor-bousfield-hL}\leavevmode
 \begin{itemize}
  \item[(a)] If $X$ is $\{p\}$-local and $Y$ is $p$-complete then
   $[X,Y]=0$ 
  \item[(b)] Conversely, if $X$ is such that $[X,Y]=0$ for all
   $p$-complete $Y$, then $X$ is $\{p\}$-local.
  \item[(c)] Similarly, if $Y$ is such that $[X,Y]=0$ for all
   $\{p\}$-local $X$, then $Y$ is $p$-complete.
  \item[(d)] If $X\in\CB$ and $Y\in\hCB_p$ then the map
   $\tht^*\:[\hL_pX,Y]\to[X,Y]$ is an isomorphism.
  \item[(e)] Thus, $\hCB_p$ is the localisation of $\CB$ with
   respect to $S/p$ (in the sense of Bousfield).  The
   corresponding localisation functor is $\hL_p$, and the following
   acyclisation functor is $\hC_p$.
 \end{itemize}
\end{corollary}
\begin{proof}\leavevmode
 \begin{itemize}
  \item[(a)] If $X$ is $\{p\}$-local and $Y$ is $p$-complete then
   $X\simeq S[1/p]\ot X$ and $Y\simeq \uHom(S^{-1}/p^\infty,X)$ so 
   \[ [X,Y] \simeq [S[1/p]\ot X,\uHom(S^{-1}/p^\infty,Y)] \simeq
       [S^{-1}/p^\infty\ot S[1/p]\ot X,Y].
   \]
   Here $S^{-1}/p^\infty\ot S[1/p]=0$, so $[X,Y]=0$.
  \item[(b)] If $X$ is as described then the natural map
   $\tht\:X\to\hL_pX$ is zero, so the map $\xi\:\hC_pX\to X$ is a
   split epimorphism, and $\hC_pX$ is $\{p\}$-local so $X$ is
   $p$-local.
  \item[(c)] If $Y$ is as described, then the natural map
   $\xi\:\hC_pY\to Y$ is zero, so the natural map $\tht\:Y\to\hL_pY$
   is a split monomorphism, and $\hL_pY$ is $p$-complete, so $Y$ is
   $p$-complete. 
  \item[(d)] The cofibration $\hC_pX\to X\xra{\tht}\hL_pX$ gives an exact
   sequence
   $[\Sg\hC_pX,Y]\to [\hL_pX,Y]\xra{\eta^*}[X,Y]\to [\hC_pX,Y]$
   in which the first and last terms are zero by~(a), so $\tht^*$
   is an isomorphism. 
  \item[(e)] We noted in Definition~\ref{defn-p-complete} that $X$ is
   $\{p\}$-local iff $X\ot S/p=0$, which means that $X$ is
   $S/p$-acyclic in the sense of Bousfield.  Given this observation,
   the required axioms are precisely what we have proved above.
 \end{itemize}
\end{proof}

\begin{remark}
 It follows from Corollary~\ref{cor-bousfield-hL} that the category
 $\hCB_p$ is again a stable homotopy category in the sense
 of~\cite{hopast:ash}. 
\end{remark}

\begin{lemma}\label{lem-same-mod-p}
 For any $X$ we have $X_p/p=\hL_p(X)/p\simeq X/p$.
\end{lemma}
\begin{proof}
 First, $X_p/p$ is just alternative notation for $\hL_p(X)/p$.
 The functor $\hL_p$ preserves the cofibration $X\xra{p.1_X}X\to X/p$,
 so $X_p/p=(X/p)_p$.  However, Example~\ref{eg-moore-exponent} shows
 that $S/p$ is uniformly $p$-torsion, so the same is true for the
 spectrum $X/p=S/p\ot X$, so $(X/p)_p=X/p$.
\end{proof}

\begin{lemma}\label{lem-X-mod-p}
 For a $p$-complete spectrum $X$, we have $X=0$ if and only if the
 spectrum $X/p=S/p\ot X$ is zero.  Thus, a map $f\:X\to Y$ of
 $p$-complete spectra is an equivalence if and only if the induced map
 $X/p\to Y/p$ is an equivalence.
\end{lemma}
\begin{proof}
 If $X/p=0$ then $X$ is $\{p\}$-local as well as $p$-complete, so
 $[X,X]=0$ by Corollary~\ref{cor-bousfield-hL}(a), so $1_X=0$, so
 $X=0$.  For the second claim, apply the first claim to the cofibre of
 $f$. 
\end{proof}

\begin{corollary}\label{cor-H-mod-p}
 Let $X$ be a bounded below $p$-complete spectrum with
 $H_*(X;\F_p)=0$; then $X=0$.  Thus, if $f\:X\to Y$ is a map of
 bounded below $p$-complete spectra, and the induced map
 $H_*(X;\F_p)\to H_*(Y;\F_p)$ is an isomorphism, then $f$ is an
 equivalence. 
\end{corollary}
\begin{proof}
 Note that 
 \[ H_*(X;\F_p)=\pi_*(H/p\ot X)=
      \pi_*(H\ot S/p\ot X) = H_*(X/p).
 \]
 Thus, if $H_*(X;\F_p)=0$ then $X/p$ is a bounded below spectrum with
 $H_*(X/p)=0$, so $X/p=0$ by the Hurewicz Theorem, so $X=0$ by
 Lemma~\ref{lem-X-mod-p}.  For the second claim, apply the first claim
 to the cofibre of $f$.
\end{proof}

\begin{proposition}\label{prop-tensor-Zp}
 For finite spectra $X$ and $Y$ there are natural isomorphisms
 \[ [X,\hL_pY] \simeq [\hL_pX,\hL_pY] \simeq [X,Y]\ot\Z_p. \] 
\end{proposition}
\begin{proof}
 We first claim that for finite $Z$ we have
 $\pi_*(\hL_pZ)=\pi_*(Z)\ot\Z_p$.  Indeed, $\pi_i(\hL_pZ)$ has a
 natural structure as a $\Z_p$-module, so the natural map
 $\pi_i(Z)\to\pi_i(\hL_pZ)$ extends to give a natural map
 $\phi_Y\:\Z_p\ot\pi_i(Z)\to\pi_i(\hL_pZ)$.  Each group $\pi_i(Z)$ is
 finitely generated, and thus is a finite direct sum of cyclic groups.
 It follows by standard calculations that $L_1(\pi_i(Z))=0$ and
 $L_0(\pi_i(Z))=\Z_p\ot\pi_i(Z)$, so we can use Lemma~\ref{lem-pi-hL}
 to see that $\phi_Z$ is an isomorphism. 

 Now suppose we have a pair of finite spectra $X$ and $Y$.
 Corollary~\ref{cor-bousfield-hL}(d) gives 
 \begin{align*}
    [\hL_pX,\hL_pY] \simeq &
    [X,\hL_pY] = [X,\uHom(S^{-1}/p^\infty,Y)] = 
    [X\ot S^{-1}/p^\infty,Y] \\
   =& 
    \pi_0\uHom(S^{-1}/p^\infty,\uHom(X,Y)) = 
    \pi_0\hL_p\uHom(X,Y).
 \end{align*}
 Here $\uHom(X,Y)$ is again dualisable and therefore finite, so 
 we can take $Z=\uHom(X,Y)$ to get 
 \[ \pi_0\hL_p\uHom(X,Y)=\Z_p\ot\pi_0\uHom(X,Y) =
    \Z_p\ot [X,Y].
 \]
\end{proof}

\begin{construction}\label{cons-CW-step-complete}
 Let $X$ be a $p$-complete spectrum with $\pi_i(X)=0$ for $i<b$.  Using the
 cofibration $X\xra{p.1_X}X\to X/p$ with the fact that
 $\pi_{b-1}(X)=0$ we get $\pi_b(X/p)=\pi_b(X)/p$.  By the Hurewicz
 Theorem, this group is also isomorphic to $H_b(X/p)=H_b(X;\F_p)$.
 Choose a family of elements $(u_t)_{t\in T}$ in $\pi_b(X)$ which
 project to give a basis for $\pi_b(X)/p$ over $\F_p$.  These elements
 assemble to give a map $u'\:T_+\ot S^b\to X$, and we can apply
 $\hL_p$ to get a map $u$ from the spectrum $W=\hL_p(T_+\ot S^b)$
 to $X$.  We define $Y$ to be the cofibre of $u$.
\end{construction}

\begin{lemma}\label{lem-CW-step-complete}
 With notation as in Construction~\ref{cons-CW-step-complete}:
 \begin{itemize}
  \item[(a)] For $i\leq b$ we have $\pi_i(Y)=0$ and $H_i(Y;\F_p)=0$.
  \item[(b)] For $i>b$ we have $H_i(Y;\F_p)=H_i(X;\F_p)$.
 \end{itemize}
\end{lemma}
\begin{proof}\leavevmode
 First, the cofibration $W\to X\to Y$ gives a cofibration
 $W/p\to X/p\to Y/p$, and $W/p=T_+\ot S^b/p$ by
 Lemma~\ref{lem-same-mod-p}.  The map $\pi_b(W/p)\to\pi_b(X/p)$ is an
 isomorphism by construction, so $\pi_i(Y/p)=0$ for $i\leq b$.  
 Next recall that $H_*(X;\F_p)=H_*(X/p)$.
 The cofibre sequence $W/p\to X/p\to Y/p$ gives a long exact sequence
 of homology groups, in which the group $H_b(W/p)=\pi_b(W/p)$ maps by
 an isomorphism to $H_b(X/p)=\pi_b(X/p)$ by construction, and all
 other homology groups of $W/p$ are trivial.  Claim~(b) follows from
 this. 
\end{proof}

\begin{definition}\label{defn-CW-complete}
 Let $X$ be a bounded below $p$-complete spectrum.  By a
 \emph{$p$-complete CW structure} on $X$ we mean a diagram
 \[ \dotsb X_{-2} \to X_{-1} \to X_0 \to X_1 \to X_2 \to \dotsb \to X
 \]
 such that 
 \begin{itemize}
  \item[(a)] $X_i=0$ for $i\ll 0$.
  \item[(b)] The cofibre $X_i/X_{i-1}$ is (equivalent to) a spectrum
   of the form $\hL_p((T_i)_+\ot S^i)$ for some set $T_i$.
  \item[(c)] The cofibre $X/X_i$ has $\pi_j(X/X_i)=0$ for $j\leq i$.  
 \end{itemize}
 We say that the structure is of \emph{finite type} if each of the
 sets $T_i$ is finite.  We say that the structure is \emph{finite} if
 in addition we have $T_i=\emptyset$ for almost all $i$.
\end{definition}

\begin{proposition}\label{prop-CW-complete}\leavevmode
 Any bounded below $p$-complete spectrum $X$ admits a $p$-complete CW
 structure
 $\{X_n\}_{n\in\Z}$ (with $X_n/X_{n-1}=\hL_p((T_n)_+\ot S^n)$ say) such
 that the following additional properties are satisfied:
 \begin{itemize}
  \item[(a)] If $\pi_i(X)=0$ for all $i<b$, then $X_i=0$ and
   $T_i=\emptyset$ for all $i<b$. 
  \item[(b)] If $H_i(X;\F_p)=0$ for all $i>t$ then $X_i=X$ for
   $i\geq t$, and $T_i=\emptyset$ for $i>t$.
  \item[(c)] If $H_i(X;\F_p)$ is finite for all $i$ then $T_i$ is
   finite for all $i$.
  \item[(d)] If $\bigoplus_iH_i(X;\F_p)$ is finite, then
   $\coprod_iT_i$ is finite.
 \end{itemize}
\end{proposition}
\begin{proof}
 If $X=0$ then everything is clear (with $X_n=0$ and $T_n=\emptyset$
 for all $n$).  Otherwise, let $b$ be the smallest integer such that
 $\pi_b(X)\neq 0$.  Put $X^k=X$ for $k\leq b$.  Then apply
 Construction~\ref{lem-CW-step-complete} repeatedly to give cofibrations
 $W_k\to X^k\to X^{k+1}$ where $W_k=\hL_p((T_k)_+\ot S^k)$.  Let
 $X^\infty$ be the homotopy colimit of the spectra $X^k$.  Then $\pi_i(X^\infty)$
 is the colimit of the groups $\pi_i(X^k)$, which are zero for $k>i$,
 so $\pi_i(X^\infty)=0$ for all $i$, so $X^\infty=0$.  Now put
 $X_i=\fib(X\to X^{i+1})$ as in Construction~\ref{cons-flip}.  The
 homotopy colimit of the spectra $X_i$ is then
 $\fib(X^{-\infty}\to X^\infty)=\fib(X\to 0)=X$, and
 $X_i/X_{i-1}=W_i$.  It follows that this gives a $p$-complete CW
 structure.  Property~(a) is clear by construction.  

 Using Lemma~\ref{lem-CW-step-complete} we see that $H_i(X^k;\F_p)=0$
 for $i<k$ and $H_i(X^k;\F_p)=H_i(X;\F_p)$ for $i\geq k$, so
 $H_i(X_k;\F_p)=H_i(X;\F_p)$ for $i\leq k$ and $H_i(X_k;\F_p)=0$ for
 $i>k$.  In particular, if $H_i(X;\F_p)=0$ for $i>t$ then when
 $k\geq t$ we have $H_*(X^{k+1};\F_p)=0$, which means that $X^{k+1}=0$
 by Corollary~\ref{cor-H-mod-p}, so $X_k=X$.  Claim~(b) follows from this. 

 Claim~(c) is clear given the above remarks about $H_*(X^k;\F_p)$, and
 claim~(d) follows from~(b) and~(c). 
\end{proof}

We would like to prove that if $X$ is bounded below and $p$-complete,
and $\bigoplus_iH_i(X;\F_p)$ is finite, then $X$ is equivalent to the
$p$-completion of a finite spectrum.  For this we will need a
construction that is similar to
Construction~\ref{cons-CW-step-complete}, but slightly different.
It will rely on the following algebraic result.

\begin{lemma}\label{lem-ABCD}
 Suppose we have an exact sequence $A\to B\xra{f}C\to D$ in which $A$
 and $D$ are finite, and $L_0(B)$ is finitely generated over $\Z_p$, and
 $L_1(B)=0$.  Then $L_0(C)$ is finitely generated over $\Z_p$, and
 $L_1(C)=0$. 
\end{lemma}
\begin{proof}
 We have short exact sequences $\ker(f)\mra B\era\img(f)$ and
 $\img(f)\mra C\era\cok(f)$ in which $\ker(f)$ is a quotient of $A$
 and $\cok(f)$ is a subgroup of $D$, so both are finite.  If $U$ is
 finite then it is easy to see that $L_1(U)=0$ and $L_0(U)$ is the
 $p$-torsion summand in $U$.  We have a six term exact sequence
 \[ L_1(\ker(f)) \to L_1(B) \to L_1(\img(f)) \to 
    L_0(\ker(f)) \to L_0(B) \to L_0(\img(f)).
 \]
 We have $L_1(B)=0$ by assumption, so $L_1(\img(f))$ embeds in the
 finite $p$-group $L_0(\ker(f))$.  However,
 $L_1(\img(f))=\Hom(\Z/p^\infty,\img(f))$, and multiplication by $p$
 is surjective on $\Z/p^\infty$, so multiplication by $p$ is injective
 on $L_1(\img(f))$.  This can only be compatible if $L_1(\img(f))=0$.
 We also see from the above sequence that $L_0(\img(f))$ is a quotient
 of $L_0(B)$ and so is finitely generated over $\Z_p$.  Next, we have
 another six term exact sequence 
 \[ L_1(\img(f)) \mra L_1(C) \to L_1(\cok(f)) \to 
    L_0(\img(f)) \to L_0(C) \to L_0(\cok(f)).
 \]
 As $\cok(f)$ is finite we have $L_1(\cok(f))=0$, and we showed above
 that $L_1(\img(f))=0$, so we have $L_1(C)=0$.  Thus, the above
 sequence reduces to a short exact sequence
 $L_0(\img(f))\mra L_0(C)\era L_0(\cok(f))$ in which $L_0(\img(f))$ is
 finitely generated over $\Z_p$ and $L_0(\cok(f))$ is a finite abelian
 $p$-group.  It follows that $L_0(C)$ is finitely generated over
 $\Z_p$, as required.
\end{proof}

\begin{definition}\label{defn-p-moderate}
 We say that a spectrum $X$ is \emph{$p$-moderate} if 
 \begin{itemize}
  \item[(a)] For all $i\ll 0$ we have $\pi_i(X)=0$ (i.e. $X$ is
   bounded below).
  \item[(b)] For all $i$ we have $L_1(\pi_i(X))=0$.
  \item[(c)] For all $i$ the group $L_0(\pi_i(X))$ is a a finitely
   generated module over $\Z_p$.
 \end{itemize}
\end{definition}

\begin{lemma}\label{lem-complete-moderate}
 Suppose that $X$ is $p$-complete and bounded below, and $H_i(X;\F_p)$
 is finite for all $i$.  Then $X$ is $p$-moderate.
\end{lemma}
\begin{proof}
 We have $\pi_i(X)=0$ for $i\ll 0$ by assumption, and
 $L_1(\pi_i(X))=0$ because $X$ is assumed to be $p$-complete, so we
 just need to check(c).  We can choose a $p$-complete CW structure as in
 Proposition~\ref{prop-CW-complete}, with each set $T_k$ being finite.
 It follows that each spectrum $X_k$ lies in the thick subcategory
 generated by $S_p$.  Here $\pi_0(S_p)=\Z_p$ and $\pi_i(S_p)=0$ for
 $i<0$ and $\pi_i(S_p)$ is finite for $i>0$, so $S_p$ is $p$-moderate.
 By a thick subcategory argument, all the homotopy groups $\pi_i(X_k)$
 are finitely generated over $\Z_p$.  We also have
 $\pi_i(X)=\pi_i(X_{i+1})$, so $\pi_i(X)$ is finitely generated as
 requred. 
\end{proof}

\begin{lemma}\label{lem-p-moderate}
 If $X$ is $p$-moderate then: 
 \begin{itemize}
  \item[(a)] For all $i$ and $e\in\N$ we have
   $\pi_i(X_p)=L_0(\pi_i(X))$ and $\pi_i(X)/p^e=\pi_i(X_p)/p^e$ and
   $\pi_i(X)[p^e]=\pi_i(X_p)[p^e]$.  
  \item[(b)] Moreover, the group
   $T=\pi_i(X)[p^\infty]=\pi_i(X_p)[p^\infty]$ is a finite
   abelian $p$-group, and $\pi_i(X_p)\simeq \Z_p^r\oplus T$ for some
   $r\geq 0$.
  \item[(c)] The natural map $X/p\to X_p/p$ is an equivalence, and we
   have short exact sequences 
   \[ \pi_i(X)/p = \pi_i(X_p)/p \mra \pi_i(X/p) \era
        \pi_{i-1}(X)[p]=\pi_{i-1}(X_p)[p].
   \]
 \end{itemize}
\end{lemma}
\begin{proof}\leavevmode
 \begin{itemize}
  \item[(a)] For any spectrum $X$ Lemma~\ref{lem-pi-hL} gives short
   exact sequences $L_0(\pi_i(X))\mra\pi_i(X_p)\era L_1(\pi_{i-1}(X))$,
   but we are assuming that the $L_1$ groups are trivial, so
   $L_0(\pi_i(X))=\pi_i(X_p)$.  The remaining claims then follow from
   Lemma~\ref{lem-A-pe}.
  \item[(b)] By assumption, the group $\pi_i(X_p)=L_0(\pi_i(X))$ is a
   finitely generated module over the discrete valuation ring $\Z_p$.
   Any such module $M$ is a finite direct sum of modules
   $\Z_p$ or $\Z/p^v$, and so has the form $\Z_p^r\oplus T$ for some
   $r\geq 0$ and some finite abelian $p$-group $T$, which is clearly
   the same as $M[p^\infty]$.
  \item[(c)] The first claim is Lemma~\ref{lem-same-mod-p}.  We can
   use the cofibration $X\xra{p.1_X}X\to X/p$
   together with~(a) to obtain the short exact sequences.
 \end{itemize}
\end{proof}

\begin{construction}\label{cons-alt-step}
 Let $X$ be a $p$-moderate spectrum such the group
 $\pi_i(X_p)=L_0(\pi_i(X))$ is zero for $i<b$.
 Put $T=\pi_b(X)[p^\infty]$, so $T$ is a finite abelian $p$-group and
 maps isomorphically to $\pi_b(X_p)[p^\infty]$.  Choose elements
 $x_0,\dotsc,x_{r-1}\in\pi_b(X_p)$ giving an isomorphism
 $\Z_p^r\oplus T\to\pi_b(X_p)$.  As $\pi_i(X_p)/p=\pi_i(X)/p$ by
 Lemma~\ref{lem-p-moderate}(c), we can choose elements $u_i\in\pi_b(X)$
 that map to $x_i$ in $\pi_b(X_p)/p$.  Put $A=\Z^r\oplus T$, so the
 elements $u_i$ give a map $f_0\:A\to\pi_b(X)$ of abelian groups.
 Lemma~\ref{lem-moore-maps} then allows us to choose a map
 $f\:\Sg^bSA\to X$ with $\pi_b(f)=f_0$.  We write $Y$ for the cofibre
 of $f$.
\end{construction}

\begin{lemma}\label{lem-alt-step}
 With notation as in Construction~\ref{cons-alt-step}:
 \begin{itemize}
  \item[(a)] $Y$ is $p$-moderate
  \item[(b)] The group $\pi_i(Y_p)=L_0(\pi_i(Y))$ is zero for
   $i\leq b$.
  \item[(c)] If $H_i(X;\F_p)=0$ for $i>t$, then also $H_i(Y;\F_p)=0$
   for $i>t$.
 \end{itemize}
\end{lemma}
\begin{proof}
 By applying the functor $\Sg^{-b}$ everywhere we can reduce to the
 case where $b=0$, which will simplify the notation.  We will freely
 make use of facts proved in Lemma~\ref{lem-p-moderate}.

 It is easy to check that $L_0(A)=\Z_p^r\oplus T$ and $L_1(A)=0$ and
 that the map
 $L_0(f_0)\:L_0(A)=\pi_0((SA)_p)\to\pi_0(X_p)\simeq\Z_p^r\oplus T$ is
 an isomorphism mod $p$.  As the source and target are
 $\Ext$-$p$-complete groups, Lemma~\ref{lem-detect-epi} tells us that
 this map is surjective.  As the source and target have the same rank,
 the kernel must be finite and so contained in $T$, but $L_0(f_0)$ is
 injective on $T$ by construction, so $L_0(f_0)$ is an isomorphism.
 It follows that the map $A\to\pi_0(X_p)$ is injective, and it factors
 through $\pi_0(f)\:A\to\pi_0(X)$, so $\pi_0(f)$ must be injective.
 As $\pi_{-1}(SA)=0$, this gives a short exact sequence
 \[ A = \pi_0(SA) \mra \pi_0(X) \era \pi_0(Y), \]
 and thus an exact sequence
 \[ L_1(A)\mra L_1(\pi_0(X)) \to L_1(\pi_0(Y)) \to 
     L_0(A) \xra{\simeq} L_0(\pi_0(X)) \era L_0(\pi_0(Y)).
 \]
 We have $L_1(A)=0$ by calculation, and $L_1(\pi_0(X))=0$ by
 assumption, and the map $L_0(A)\to L_0(\pi_0(X))$ is an isomorphism
 by construction.  It follows that $L_1(\pi_0(Y))=L_0(\pi_0(Y))=0$.
 It is standard that for $i>0$ the group $\pi_i(SA)$ is finite, and
 we have an exact sequence
 \[ \pi_i(SA) \to \pi_i(X) \to \pi_i(Y) \to
      \pi_{i-1}(SA) \to \pi_{i-1}(X)
 \]
 If $i>1$ then $\pi_i(SA)$ and $\pi_{i-1}(SA)$ are both finite so
 Lemma~\ref{lem-ABCD} immediately implies that $L_0(\pi_i(Y))$ is
 finitely generated over $\Z_p$ and $L_1(\pi_i(Y))=0$.  For $i=1$ the
 group $\pi_{i-1}(SA)=\pi_0(SA)=A$ need not be finite, but we have
 seen that the map $A\to\pi_0(X)$ is injective, so we have an exact
 sequence $\pi_1(SA)\to\pi_1(X)\to\pi_1(Y)\to 0$, and we get the same
 conclusion.  This completes the proof that $Y$ is $p$-moderate.
 
 Now suppose that the groups $H_i(X;\F_p)=H_i(X/p)$ are zero for
 $i>t$; we need to show that $Y$ has the same property.  Put
 $W=SA\ot S/p=SA_p\ot S/p$, so we have a cofibration 
 $W\to X/p\to Y/p$, with $\pi_i(X/p)=0$ for $i<0$ and $\pi_i(Y/p)=0$
 for $i\leq 0$.  It is straightforward to check that $H_0(W)=A/p$ and
 $H_1(W)=A[p]$ and all other homology groups are trivial.  It follows
 that $H_i(Y/p)=H_i(X/p)$ for all $i\geq 3$, which is enough whenever
 $t\geq 2$.  

 Now consider the case where $t=1$.  We have a diagram as follows.
 \begin{center}
  \begin{tikzcd}
   0=H_2(X/p) \arrow[r] &
   H_2(Y/p) \arrow[r] &
   H_1(W) \arrow[r,rightarrowtail] \arrow[d,rightarrowtail,"\phi"'] & 
   H_1(X/p) \arrow[r] \arrow[d,"\psi"] & 
   H_1(Y/p) \\
   & & \pi_0(SA_p) \arrow[r,"\simeq"] & 
   \pi_0(X_p)
  \end{tikzcd}
 \end{center}
 The top row arises from the cofibration $W\to X/p\to Y/p$ and so is
 exact.  The top left group $H_2(X/p)=H_2(X;\F_p)$ is zero by
 assumption because $2>t=1$.  Next recall that we have a cofibration
 $S\xra{p}S\xra{}S/p\xra{\bt}S^1$.  This gives a map
 $1_{SA_p}\ot\bt\:W=SA_p\ot S/p\to\Sg SA_p$ and thus a map
 $H_1(W)\to H_0(SA_p)$.  However, we also have a Hurewicz isomorphism
 $A_p=\pi_0(SA_p)\to H_0(SA_p)$, so we get a map
 $H_1(W)\to\pi_0(SA_p)$, which is the map marked $\phi$ in the above
 diagram.   The kernel is $H_1(SA)/p=0$, so $\phi$ is injective.  The
 map $\psi$ is produced in the same way starting from
 $1_{X_p}\ot\bt\:X/p\to\Sg X_p$, and this construction is natural
 so we have a commutative square as indicated.  The bottom map is an
 isomorphism by construction, so the map $H_1(W)\to H_1(X/p)$ must be
 injective, so $H_2(Y/p)=0$ as required.

 Finally consider the case where $k=0$, so $H_1(X/p)=0$.  The
 cofibration $W\to X/p\to Y/p$ gives an exact sequence
 $0\to H_1(Y/p)\to H_0(W)\to H_0(X/p)$, but the Hurewicz Theorem gives
 $H_0(W)=H_0(X/p)=A/p$, so $H_1(Y/p)=0$ as required.
\end{proof}

\begin{remark}
 Just as in Remark~\ref{rem-coprod}, the subcategory $\hCB_p\sse\CB$ is
 not closed under infinite coproducts, but $\hCB_p$ has its own
 coproducts, given by taking the coproduct in $\CB$ and then applying
 the functor $\hL_p$.  An object $X\in\hCB_p$ is (by definition)
 compact if and only if the functor $[X,-]$ sends these adjusted
 coproducts to direct sums.  Similarly, $\hCB_p$ has a closed symmetric
 monoidal structure, with the monoidal unit being $S_p$, monoidal
 product given by $(X,Y)\mapsto\hL_p(X\ot Y)$ and the internal
 mapping objects being the same as in $\CB$.  Thus, an object
 $X\in\hCB_p$ is strongly dualisable if and only if the natural map
 $\hL_p(\uHom(X,S_p)\ot X)\to \uHom(X,X)$ is an equivalence.
\end{remark}

\begin{proposition}\label{prop-finite-complete}
 Let $X$ be a bounded below $p$-complete spectrum.  Then the following
 are equivalent: 
 \begin{itemize}
  \item[(a)] $X$ admits a finite $p$-complete CW structure.
  \item[(b)] $X$ lies in the thick subcategory generated by $S_p=\hL_pS$.
  \item[(c)] $X$ is strongly dualisable in $\hCB_p$.
  \item[(d)] $X/p$ is a finite spectrum.
  \item[(e)] $\bigoplus_iH_*(X;\F_p)$ is finitely generated over $\F_p$.
  \item[(f)] There is a finite spectrum $X'$ with $X\simeq \hL_pX'$.
 \end{itemize}
\end{proposition}
\begin{proof}
 It is clear that~(a) implies~(b).  The category of strongly
 dualisable objects is always thick, and contains the unit object
 $S_p$, so~(b) implies~(c).  Suppose that $X$ is strongly dualisable,
 so $\hL_p(\uHom(X,S_p)\ot X)=\uHom(X,X)$.  Using
 Lemma~\ref{lem-same-mod-p} we deduce that
 $\uHom(X/p,S)\ot X/p=\uHom(X/p,X/p)$, so $X/p$ is strongly dualisable in
 $\CB$, and so is a finite spectrum by
 Proposition~\ref{prop-finite-global}.  This shows that~(c)
 implies~(d).  As $H_*(X;\F_p)=H_*(X/p)$, we see that~(d) implies~(e).  

 Now suppose that~(e) holds.  Lemma~\ref{lem-complete-moderate}
 implies that $X$ is $p$-moderate.  Let $b$ be the smaallest integer
 such that $\pi_b(X)\neq 0$, and let $t$ be the largest integer such
 that $H_t(X;\F_p)\neq 0$.  Define $X^k=X$ for $k\leq b$.  Use
 Construction~\ref{cons-alt-step} repeatedly to produce cofibrations
 $\Sg^kSA_k\to X^k\to X^{k+1}$ with $A_k$ being a finitely generated
 abelian group, such that $X^k$ is always $p$-moderate with
 $L_0(\pi_i(X^k))=0$ for $i<k$ and $H_i(X^k;\F_p)=0$ for $i>t$.  This
 means that $H_*(X^{t+1};\F_p)=0$, so $(X^{t+1})_p=0$.  Let $X_k$ be
 the fibre of the map $X\to X^{k+1}$.  We find that $X_{b-1}=0$ and
 $X_k/X_{k-1}=\Sg^kSA_k$, so each $X_k$ is a finite spectrum.  In
 particular, $X_t$ is a finite spectrum, and the cofibre
 $X^{t+1}=X/X_t$ has $(X/X_t)_p=0$, so the induced map $(X_t)_p\to X$
 is an equivalence.  This shows that~(e) implies~(f).  Finally, if~(f)
 holds then we can apply $\hL_p$ to a finite CW structure on $X'$ to
 obtain a finite $p$-complete CW structure on $X$, so~(a) holds.
\end{proof}

\begin{proposition}\label{prop-small-complete}
 Let $X$ be a bounded below $p$-complete spectrum.  Then the following
 are equivalent: 
 \begin{itemize}
  \item[(p)] $X$ lies in the thick subcategory generated by $S/p$
  \item[(q)] $X$ is a finite spectrum with $p^n.1_X=0$ for some $n$. 
  \item[(r)] $X$ is compact in $\hCB_p$
 \end{itemize}
\end{proposition}
\begin{proof}
 If we have a cofibration $X\to Y\to Z$ with $p^n.1_X=0$ and
 $p^m.1_Z=0$ then it is not hard to check that $p^{n+m}.1_Y=0$.  Using
 this, we see that the category of spectra satisfying~(q) is thick.
 It contains $S/p$ by Example~\ref{eg-moore-exponent}, and using this
 we see that~(p) implies~(q).  For the converse, let $\CT$ be the
 thick subcategory generated by $S/p$.  There are cofibrations
 $S/p\to S/p^{k+1}\to S/p^k$ (corresponding to the obvious short exact
 sequences $\Z/p\mra \Z/p^{k+1}\era\Z/p^k$) and we can use these
 inductively to see that $S/p^n\in\CT$ for all $n$.  The category
 $\{X\st X\ot S/p^n\in\CT\}$ is thick and contains $S^0$, so it
 contains all finite spectra.  If $p^n.1_X=0$ then the cofibration
 $X\xra{p^n.1_X=0}X\to S/p^n\ot X$ shows that $X$ is a retract
 of $S/p^n\ot X$ and so lies in $\CT$.  From this we see that~(q)
 implies~(p).  We can apply the Spanier-Whitehead duality functor
 $D(-)=\uHom(-,S)$ to the cofibration 
 \[ S \xra{p} S \xra{} S/p \to S^1 \xra{p} S^1 \]
 to see that the dual of $S/p$ is $\Sg^{-1}S/p$.  Thus,
 using Lemma~\ref{lem-same-mod-p} we get 
 \[ \textstyle
    \left[S/p,\hL_p\left(\bigoplus_iX_i\right)\right] = 
    \left[S^1,S/p\ot\hL_p\left(\bigoplus_iX_i\right)\right] = 
    \left[S^1,S/p\ot\bigoplus_iX_i\right] = 
    \left[S/p,\bigoplus_iX_i\right] = \bigoplus_i[S/p,X_i].
 \]
 This shows that $S/p$ is compact in $\hCB_p$, so everything in the
 thick subcategory generated by $S/p$ is also compact, so~(p)
 implies~(r).  Conversely, suppose that $X$ is compact in $\hCB_p$.  As
 the functor $[X,-]$ preserves coproducts, it also sends sequential
 homotopy colimits to colimits.  The homotopy colimit of the sequence
 $X\xra{p}X\xra{p}X\to\dotsb$ in $\hCB_p$ is $\hL_p(S[1/p]\ot X)=0$,
 so the colimit of the sequence
 $[X,X]\xra{p}[X,X]\xra{p}[X,X]\to\dotsb$ is zero.  In particular the
 element $1_X$ in the first of these groups must map to zero in the
 colimit, which means that $p^n.1_X=0$ for some $n$.  Thus, (r)
 implies~(q).
\end{proof}

\begin{corollary}\label{cor-F-prime}
 Let $\CF$ be the category of finite spectra and let $\CF_p$ denote
 the category of strongly dualisable objects in $\hCB_p$.  Let $\CF'_p$
 be the category with the same objects as $\CF$, but with
 $\CF'_p(X,Y)=\CF(X,Y)\ot\Z_p$.  Then $\CF'_p$ is equivalent to
 $\CF_p$ (and so has a natural triangulation).
\end{corollary}
\begin{proof}
 Proposition~\ref{prop-tensor-Zp} shows that $\hL_p$ induces a full
 and faithful functor $\CF'_p\to\CF_p$.  The fact that~(c) implies~(f)
 in Proposition~\ref{prop-finite-complete} means that this is also
 essentially surjective.
\end{proof}

\section{Modules over the $p$-complete sphere}
\label{sec-module}

Just as in Remark~\ref{rem-modules}, we can 
use~\cite[Chapter III and Theorem VIII.2.1]{ekmm:rma} or 
\cite[Section 4.5 and Theorem 7.5.0.6]{lu:ha} or some similar
framework to produce examples of the following structure.

\begin{definition}
 A \emph{strict module category} for $S_p$ consists of a
 monogenic stable homotopy category $\CM_p$ together with an
 exact, strongly symmetric monoidal functor $P\:\CB\to\CM_p$ and
 a right adjoint $U\:\CM_p\to\CB$ such that $UPX$ is naturally
 isomorphic to $S_p\ot X$.  In this situation we will write $M\ot_pN$
 for the symmetric monoidal product in $\CM_p$, and we write $S_p$ for
 the unit object (which is the same as $P(S)$ and satisfies
 $U(S_p)=S_p$).
\end{definition}

For the rest of this section we assume that we have a given strict
module category as above. 

\begin{remark}
 Because $S_p$ is not compact in $\hCB_p$ we see that $\hCB_p$ is not
 monogenic and so cannot be equivalent to $\CM_p$.  Moreover, the
 image of $U$ contains spectra such as $U(P(S\Q))=S(\Q\ot\Z_p)$, which
 is not $p$-complete, so the functor $U$ does not even take values in
 $\hCB_p$.
\end{remark}

\begin{definition}
 For $M\in\CM_p$ we define
 $\pi_i(M)=\pi_i(U(M))=\CB(S^i,U(M))=\CM_p(\Sg^iP(S),M)$. 
\end{definition}

\begin{remark}\label{rem-conservative-complete}
 Because $\CM_p$ is assumed to be monogenic, we have $M=0$ iff
 $\pi_*(M)=0$ iff $\pi_*(U(M))=0$ iff $U(M)=0$.  By applying this to
 cofibres, we see that a morphism $f\:M\to N$ is an isomorphism iff
 $\pi_*(f)$ is an isomorphism iff $U(f)$ is an isomorphism.
\end{remark}

\begin{lemma}\label{lem-pi-P}
 For $X\in\CB$ we have $\pi_*(P(X))=\Z_p\ot\pi_*(X)$.
\end{lemma}
\begin{remark}
 By definition $\pi_i(P(X))=\pi_i(U(P(X)))$, and $U(P(X))=S_p\ot X$ by
 the axioms for a strict module category.  There is an evident pairing
 $\al_X\:\Z_p\ot\pi_*(X)=\pi_0(S_p)\ot\pi_*(X)\to\pi_*(S_p\ot X)$,
 which is clearly an isomorphism when $X=S$.  As $\Z_p$ is flat, the
 domain of $\al$ is a homology theory as well as the codomain, and it
 follows that $\al_X$ is an isomorphism for all $X$.
\end{remark}

\begin{definition}\leavevmode
 \begin{itemize}
  \item[(a)] We say that $M\in\CM_p$ is $p$-complete iff each group
   $\pi_i(M)$ is $\Ext$-$p$-complete iff $U(M)$ is $p$-complete.
   We write $\hCM_p$ for the full subcategory of $p$-complete objects
   in $\CM_p$.
  \item[(b)] We say that $M$ is $\{p\}$-local iff each group
   $\pi_i(M)$ is $\{p\}$-local iff $p.1_M$ is an isomorphism iff
   $U(M)$ is $\{p\}$-local iff the object $M/p=\cof(p.1_M)$ is zero. 
  \item[(c)] We define $\hC_p,\hL_p\:\CM_p\to\CM_p$ by
   $\hC_p(M)=\uHom(P(S[1/p]),M)$ and $\hL_p(M)=\uHom(P(S^{-1}/p^\infty),M)$.
  \item[(d)] We define $\hP=\hL_p\circ P\:\CB\to\CM_p$.
 \end{itemize}
\end{definition}

\begin{remark}\label{rem-moore-module}
 As $S_p=P(S)$ is the unit object and tensor products preserve
 cofibrations we see that the object $M/p=\cof(p.1_M)$ mentioned above
 can also be identified with $P(S/p)\ot_pX$.
\end{remark}

\begin{lemma}\label{lem-enriched-adjoint}
 For $T\in\CB$ and $M\in\CM_p$ we have a natural isomorphism 
 $\uHom(T,UM)\simeq U\uHom(PT,M)$ in $\CB$.
\end{lemma}
\begin{proof}
 For $X\in\CB$ we have natural isomorphisms
 \begin{align*}
  \CB(X,\uHom(T,UM)) =&
  \CB(X\ot T,UM) =
  \CM_p(P(X\ot T),M) = \\ 
  & \CM_p(PX\ot_p PT,M) = 
  \CM_p(PX,\uHom(PT,M)) = 
  \CB(X,U\uHom(PT,M)).
 \end{align*}
 The claim follows by the Yoneda Lemma.
\end{proof}

\begin{example}\label{eg-enriched-adjoint}
 By taking $T=S^{-1}/p^\infty$  or $T=S[1/p]$ in the lemma we see that
 $U\hL_p=\hL_pU$ and $U\hC_p=\hC_pU$, so the following diagram commutes
 up to natural isomorphism.
 \begin{center}
  \begin{tikzcd}
   \CM_p \arrow[d,"U"'] & 
   \CM_p \arrow[l,"\hL_p"'] \arrow[r,"\hC_p"] \arrow[d,"U"'] &
   \CM_p \arrow[d,"U"'] \\
   \CB &
   \CB \arrow[l,"\hL_p"] \arrow[r,"\hC_p"'] & 
   \CB
  \end{tikzcd}
 \end{center}
\end{example}

\begin{proposition}\label{prop-hL-module}\leavevmode
 \begin{itemize}
  \item[(a)] $\hL_pM$ is always $p$-complete.   
  \item[(b)] $\hC_pM$ is always $\{p\}$-local.
  \item[(c)] An object $M\in\CM_p$ is $p$-complete iff $\hC_pM=0$ iff the map
   $\tht\:M\to\hL_pM$ is an equivalence.
  \item[(d)] An object $M\in\CM_p$ is $\{p\}$-local iff $\hL_pX=0$ iff the map
   $\xi\:\hC_pX\to X$ is an equivalence.
 \end{itemize}
\end{proposition}
\begin{proof}
 Remark~\ref{rem-conservative-complete} and
 Example~\ref{eg-enriched-adjoint} reduce everything to the
 corresponding claims in Proposition~\ref{prop-hL}.
\end{proof}

\begin{lemma}\label{lem-M-pinv}
 For any $M\in\CM_p$ we have a natural map $M\to P(S[1/p])\ot_pM$
 which induces an isomorphism
 $\pi_*(M)[1/p]\to\pi_*(P(S[1/p])\ot_pM)$.  Thus, this map is an
 isomorphism if and only if $M$ is $\{p\}$-local.
\end{lemma}
\begin{proof}
 Recall from Construction~\ref{cons-S-pinv} that $S[1/p]$ can be
 constructed as the cofibre of a map $\mu_{1-pt}\:S[t]\to S[t]$, where
 $S[t]=\bigoplus_{n\in\N}S$.  The functor $X\mapsto P(X)\ot_pM$
 preserves coproducts and cofibrations, so $P(S[1/p])\ot_pM$ is the
 cofibre of $\mu_{1-pt}\:M[t]\to M[t]$, where
 $M[t]=\bigoplus_{n\in\N}M$.  As $\CM_p$ is monogenic we see that the
 functor $\pi_i\:\CM_p\to\Ab$ also preserves coproducts and sends
 cofibrations to exact sequences.  It follows that
 $\pi_i(P(S[1/p])\ot_pM)=\pi_i(M)[1/p]$ as claimed.
\end{proof}

\begin{corollary}\label{cor-bousfield-hL-module}\leavevmode
 \begin{itemize}
  \item[(a)] If $M$ is $\{p\}$-local and $N$ is $p$-complete then
   $\CM_p(M,N)=0$ 
  \item[(b)] Conversely, if $M$ is such that $\CM_p(M,N)=0$ for all
   $p$-complete $N$, then $M$ is $\{p\}$-local.
  \item[(c)] Similarly, if $N$ is such that $\CM_p(M,N)=0$ for all
   $\{p\}$-local $M$, then $N$ is $p$-complete.
  \item[(d)] If $M\in\CM_p$ and $N\in\hCM_p$ then the map
   $\tht^*\:[\hL_pM,N]\to\CM_p(M,N)$ is an isomorphism.
  \item[(e)] Thus, $\hCM_p$ is the localisation of $\CM_p$ with
   respect to $P(S/p)$ (in the sense of Bousfield).  The
   corresponding localisation functor is $\hL_p$, and the following
   acyclisation functor is $\hC_p$.
 \end{itemize}
\end{corollary}
\begin{proof}\leavevmode
 This can be proved in the same way as
 Corollary~\ref{cor-bousfield-hL}, using Lemma~\ref{lem-M-pinv} in the
 first step.
\end{proof}

\begin{proposition}\label{prop-complete-equiv}
 The functors $\hP$ and $U$ restrict to give an adjoint equivalence
 between $\hCB_p$ and $\hCM_p$.
\end{proposition}
\begin{proof}
 As $M$ is defined to be $p$-complete if and only if $U(M)$ is
 $p$-complete, we see that $U$ restricts to give a functor
 $U_1\:\hCM_p\to\hCB_p$.  As the object $\hP(X)=\hL_p(P(X))$ is always
 $p$-complete by Proposition~\ref{prop-hL-module}(a), we see that
 $\hP$ restricts to give a functor $\hP_1\:\hCB_p\to\hCM_p$.  For
 $X\in\hCB_p$ and $M\in\hCM_p$ we can use
 Corollary~\ref{cor-bousfield-hL-module}(d) to get
 \[ \hCM_p(\hP(X),M) = 
    \hCM_p(\hL_p(P(X)),M) = 
    \CM_p(P(X),M) = \CB(X,UM).
 \]
 This shows that $\hP_1$ is left adjoint to $U_1$.  For $X\in\hCB_p$
 we can use Example~\ref{eg-enriched-adjoint} to get 
 \[ U(\hP(X)) = U(\hL_p(P(X))) = \hL_p(U(P(X))) = 
     \hL_p(S_p\ot X).
 \]
 Using Lemmas~\ref{lem-same-mod-p} and~\ref{lem-X-mod-p} we see that
 the natural map $X\to S_p\ot X$ becomes an isomorphism after
 applying $\hL_p$, but $X\in\hCB_p$ so $\hL_p(X)=X$.  We thus have
 $U(\hP(X))\simeq X$, and after identifying the maps that were
 implicitly used we find that the unit $\eta\:X\to U(\hP(X))$ is an
 isomorphism.  The functor $U_1$ reflects isomorphisms, so
 Lemma~\ref{lem-adjoint-equiv} completes the proof. 
\end{proof}

\begin{definition}
 We write $\hCB_p^d$ for the subcategory of strongly dualisable
 objects in $\hCB_p$ (which has various other descriptions as in
 Proposition~\ref{prop-finite-complete}, and is equivalent to the
 category $\CF'_p$ as in Corollary~\ref{cor-F-prime}).  We write
 $\hCM_p^d$ for the subcategory of strongly dualisable objects in
 $\CM_p$.  
\end{definition}

\begin{remark}
 As $\CM_p$ is assumed to be monogenic, we can
 use~\cite[Theorem 2.1.3(d)]{hopast:ash} to see that $\CM_p^d$ is also
 the subcategory of compact objects, and is the thick subcategory
 generated by the unit object $S_p=P(S)$.
\end{remark}

\begin{proposition}\label{prop-small-modules}
 The category $\CM^d_p$ is contained in $\hCM_p$, and 
 the functors $\hP$ and $U$ restrict further to give an adjoint
 equivalence between $\hCB^d_p$ and $\CM^d_p$.  Thus, $\CM_p^d$ is
 also equivalent to the category $\CF'_p$ in in
 Corollary~\ref{cor-F-prime}.
\end{proposition}
\begin{proof}
 Put $\CC=\{M\in\CM_p\st U(M)\in\hCB_p^d\}$.  As $U(P(S))=S_p$ we see
 that $P(S)\in\CC$.  As $\CC$ is thick and $\CM_p^d$ is the thick
 subcategory generated by $P(S)$ we see that $\CM_p^d\sse\CC$, so
 $U(\CM_p^d)\sse\hCB_p^d$.  As $M$ is complete whenever $U(M)$ is, we
 see that $\CM_p^d\sse\hCM_p$.

 Similarly, we know from Proposition~\ref{prop-finite-complete} that
 $\hCB_p^d$ is the thick subcategory generated by $S_p$.  We have seen
 that $U(\hP(S_p))=S_p$, so the natural map $P(S)\to\hP(S_p)$ becomes
 an isomorphism after applying $U$, but $U$ reflects isomorphisms, so
 $\hP(S_p)$ is the unit object $P(S)$, and in particular lies in
 $\CM_p^d$.  By a thick subcategory argument, we conclude that
 $\hP(\hCB_p^d)\sse\CM_p^d$.  From
 Proposition~\ref{prop-complete-equiv} we know that the unit
 $X\to U(\hP(X))$ is an isomorphism when $X\in\hCB_p$, and in
 particular when $X\in\hCB^d_p$.  We also know that the counit
 $\hP(U(M))\to M$ is an isomorphism whenever $M\in\hCM_p$, and in
 particular when $M\in\CM_p^d$.  We thus have an adjoint equivalence
 as claimed.
\end{proof}

We now want to show that $\CM_p$ satisfies homological Brown
representability.  
\begin{definition}\label{defn-homology}
 Recall that 
 \begin{itemize}
  \item[(a)] A covariant or contravariant functor from a triangulated
   category to an abelian category is said to be \emph{exact} if it
   converts cofibre sequences to exact sequences.
  \item[(b)] A \emph{homology functor} on $\CM_p$ is an exact functor
   $F\:\CM_p\to\Ab$ that preserves all coproducts.
  \item[(c)] For any $E\in\CM_p$ we have a homology functor
   $H_E\:\CM_p\to\Ab$ given by $H_E(X)=\pi_0(E\ot_pX)$.
  \item[(d)] We say that a homology functor $F$ is
   \emph{representable} if it is isomorphic to $H_E$ for some $E$.
 \end{itemize}
\end{definition}

\begin{theorem}\label{thm-brown}
 Every homology functor $\CM_p\to\Ab$ is representable, and every
 natural transformation $H_D\to H_E$ is induced by a morphism
 $D\to E$.  In other words, $\CM_p$ is a Brown category as
 in~\cite[Definition 4.1.4]{hopast:ash}.
\end{theorem}

Our proof, which will be given after some preparatory results, is
closely related to the approach of Neeman~\cite{ne:tba}.  He considers
a triangulated category $\CS$ with only countably many morphisms, and
proves a representability result in that context.  In our case the
relevant category is $\CM_p^d$.  Every object here is isomorphic to
$FX$ for some finite spectrum $X$, which implies that there are
only countably many isomorphism classes.  Morphism sets may be
uncountable, but they are always finitely generated modules over
$\Z_p$.  We will show that this is enough to replace one step in
Neeman's argument, and the other steps can be used without further
changes, although we will present them in a somewhat different way.

Note that homology functors on $\CM_p$ are equivalent to exact
functors $\CM_p^d\to\Ab$, by~\cite[Corollary 2.3.11]{hopast:ash}.
Moreover, the duality functor $D(X)=\uHom(X,S_p)$ gives an equivalence
$(\CM_p^d)^\op\to\CM_p^d$.  Thus, everything can be reinterpreted in
terms of exact functors $(\CM_p^d)^\op\to\Ab$.

\begin{definition}\leavevmode 
 \begin{itemize}
  \item[(a)] We write $\CA$ for the category of additive functors
   $(\CM_p^d)^\op\to\Ab$.
  \item[(b)] For each $U\in\CM_p$ we define $H^U\in\CA$
   by $H^U(X)=[X,U]=\CM_p(X,U)=H_U(D(X))$.
  \item[(c)] We say that an object $P\in\CM_p$ is \emph{free} if it
   can be expressed as a coproduct $\bigoplus_{i\in I}U_i$ with
   $U_i\in\CM_p^d$, and note that in this case we have
   $H^P\simeq\bigoplus_{i\in I}H^{U_i}$.  We say that a functor
   $F\in\CA$ is \emph{free} if it is isomorphic to $H^P$ for some free
   object $P$.  We write $\CP'$ for the full subcategory of free
   objects in $\CM_p$, and $\CP$ for the category of free functors.
  \item[(d)] We write $\CH$ for the full subcategory of exact functors
   in $\CA$.
  \item[(e)] We write $\CP_\om$ for the full
   subcategory of $\CP$ consisting of objects that can be expressed as
   $\bigoplus_{i\in I}H^{U_i}$ for some countable set $I$.
  \item[(f)] We say that a functor $F\in\CA$ is countably generated if it
   admits an epimorphism from some object in $\CP_\om$, and we write
   $\CA_\om$ for the category of countably generated functors.
 \end{itemize}
\end{definition}

\begin{lemma}\label{lem-free-functors}
 If $P\in\CP'\sse\CM_p$ and $X\in\CM_p$ then the natural map
 $\CM_p(P,X)\to\CA(H^P,H^X)$ is an isomorphism.  Thus the functor
 $P\mapsto H^P$ gives an equivalence $\CP'\to\CP$. 
\end{lemma}
\begin{proof}
 We can express $P$ as $\bigoplus_{i\in I}U_i$ with $U_i\in\CM_p^d$,
 and then $\CM_p(P,X)=\prod_i\CM_p(U_i,X)$ and
 $H^P=\bigoplus_iH^{U_i}$ so $\CA(H^P,H^X)=\prod_i\CA(H^{U_i},H^X)$.
 Moreover, the Yoneda Lemma gives
 $\CA(H^{U_i},H^X)=H^X(U_i)=\CM_p(U_i,X)$.  The claim is clear from
 this. 
\end{proof}

\begin{lemma}\label{lem-two-of-three}
 Given a short exact sequence $F\mra G\era H$ in $\CA$, if two of the
 functors $F$, $G$ and $H$ lie in $\CH$, then so does the third.
\end{lemma}
\begin{proof}
 Given a cofibre sequence $X=(X_0\to X_1\to X_2\to\Sg X)$ we can
 extend the sequence infinitely in both directions in the usual way
 and then apply $F$ to get a chain complex which we call $F(X)$.  By
 definition, $F$ lies in $\CH$ iff $F(X)$ is acyclic for all $X$.  A
 short exact seequence of functors $F\mra G\era H$ gives a short exact
 sequence $F(X)\mra G(X)\era H(X)$ of chain complexes and thus a long
 exact sequence of homology groups, using which we see that if two of
 $F(X)$, $G(X)$ and $H(X)$ is acyclic, then the third is also.  The
 claim is clear from this.
\end{proof}

\begin{lemma}\label{lem-cok-rep}
 Suppose we have a short exact sequence $0\to Q\xra{d}P\xra{f}F\to 0$
 in $\CA$ with $P,Q\in\CP$.  Then we can choose an isomorphism
 $F\simeq H^T$ for some $T\in\CM_p$.  Moreover, for any $X$ the
 resulting map $\CM(T,X)\to\CA(F,H^X)$ is surjective.
\end{lemma}
\begin{proof}
 By the definition of $\CP$ we have $P=H^U$ and $Q=H^V$ for some
 $U,V\in\CP'$.  Lemma~\ref{lem-free-functors} tells us that there is
 a unique $d'\:V\to U$ that induces $d\:H^V\to H^U$.  We can
 now choose a cofibre $U\xra{f'}T$ for $d'$, and we find that $f'$
 induces an isomorphism $\cok(f)\to H^T$, so $F\simeq H^T$ as
 required.  Given another object $X\in\CM_p$ we can chase the diagram 
 \begin{center}
  \begin{tikzcd}
   \CM_p(\Sg V,X) \arrow[r] &
   \CM_p(T,X) \arrow[d] \arrow[r] & 
   \CM_p(U,X) \arrow[d,"\simeq"] \arrow[r] &
   \CM_p(V,X) \arrow[d,"\simeq"] \\ &
   \CA(F,H^X) \arrow[r,rightarrowtail] & 
   \CA(H^U,H^X) \arrow[r] &
   \CA(H^V,H^X)
  \end{tikzcd}
 \end{center}
 to see that the map $\CM_p(T,X)\to\CA(F,H^X)$ is surjective.
\end{proof}

\begin{definition}\label{defn-skel}
 Proposition~\ref{prop-small-modules} tells us that every object in
 $\CM_p^d$ is isomorphic to $P(X)$ for some finite spectrum $X$, and
 it follows that $\CM_p^d$ has only countably many isomorphism
 classes.  We choose a list $T_0,T_1,\dotsc$ containing precisely one
 representative of each isomorphism class, starting with $T_0=0$ and
 $T_1=S_p$.  We write $\CS$ for the full subcategory of $\CM_p^d$ with
 objects $\{T_i\st i\in\N\}$, so the inclusion $\CS\to\CM_p^d$ is an
 equivalence.  
\end{definition}

\begin{definition}\label{defn-well-filtered}
 Consider a functor $H\in\CH$.  We define $\el(H)$ to be the category
 of pairs $(X,x)$ with $X\in\CS$ and $x\in H(X)$.  The morphisms from
 $(X,x)$ to $(Y,y)$ are the maps $f\in\CS(X,Y)$ with $f^*(y)=x$.
 Recall that a subcategory $\CX\sse\el(H)$ is said to be
 \emph{filtered} if it has properties~(a) to~(c) below; we will say
 that it is \emph{well-filtered} if it also has property~(d).
 \begin{itemize}
  \item[(a)] $\CX$ is not empty
  \item[(b)] For all $P_0,P_1\in\CX$ there is an object $P_2\in\CX$
   and morphisms $P_0\xra{}P_2\xla{}P_1$.
  \item[(c)] For all pairs of parallel morphisms $f,g\:P_0\to P_1$ in
   $\CX$ there is a morphism $h\:P_1\to P_2$ in $\CX$ with $hf=hg$
  \item[(d)] For all $(X,x)\in\CX$ and $W\in\CS$ and
   $f\:W\to X$ with $f^*(x)=0$, there exists $g\:(X,x)\to (Y,y)$ in
   $\CX$ with $gf=0$.
 \end{itemize}
\end{definition}

\begin{proposition}\label{prop-well-filtered}
 Any countable subcategory of $\el(H)$ is contained in a countable
 well-filtered subcategory.
\end{proposition}
\begin{proof}
 Let $\CX_0$ be a countable subcategory.  We define an increasing
 sequence of countable subcategories $\CX_n$ as follows.  Given
 $\CX_n$, we define $\CX_{n+1}$ by starting with $\CX_n$ and adding
 extra objects and morphisms:
 \begin{itemize}
  \item[(a)] For each pair of objects $(X_0,x_0)$ and $(X_1,x_1)$ in
   $\CX_n$, we choose a coproduct diagram
   $X_0\xra{i_0}X_2\xla{i_1}X_1$ in $\CS$.  As $H$ is
   cohomological, there is a unique element $x_2\in H(X_2)$ with
   $i_0^*(x_2)=x_0$ and $i_1^*(x_2)=x_1$.  We add the object
   $(X_2,x_2)$ and the morphisms
   $(X_0,x_0)\xra{i_0}(X_2,x_2)\xla{i_1}(X_1,x_1)$ to $\CX_{n+1}$.
  \item[(b)] For each pair of objects 
   parallel morphisms $f,g\:(X_0,x_0)\to(X_1,x_1)$ in $\CX_n$, we
   choose a cofibre $h\:X_1\to X_2$ for $g-f$ in $\CS$.  As $H$
   is cohomological, we can choose $x_2\in H(X_2)$ with
   $h^*(x_2)=x_1$.  We add the object $(X_2,x_2)$ and the morphism
   $h\:(X_1,x_1)\to(X_2,x_2)$ to $\CX_{n+1}$.
  \item[(c)] For each $W\in\CS$ and $(X,x)\in\CX_n$, we note
   that $[W,X]$ is a finitely generated $\Z_p$-module, so the same is
   true for the submodule $\{m\in[W,X]\st m^*(x)=0\}$.  We choose
   generators $m_0,\dotsc,m_r$ for this submodule, and combine them to
   give a morphism $m\:\bigoplus_{i<r}W\to X$ with $m^*(x)=0$.  We
   then choose a cofibre $f\:X\to Y$ in $\CS$, and an element
   $y\in H(Y)$ with $f^*(y)=x$.  We add the object $(Y,y)$ and the
   morphism $f\:(X,x)\to (Y,y)$ to $\CX_{n+1}$.
  \item[(d)] Finally, we add in all composites of composable chains of
   morphisms that have already been included.
 \end{itemize}
 Standard arguments show that $\CX_{n+1}$ is still countable, as is
 the union $\CX_\infty=\bigcup_n\CX_n$, and $\CX_\infty$ is also
 well-filtered.  
\end{proof}

\begin{corollary}\label{cor-countable}
 If $H\in\CH$ and $F\leq H$ is countably generated, then there exists
 $G$ with $F\leq G\leq H$ and a short exact sequence 
 $0\to P\xra{d}P\xra{f}G\to 0$ with $P\in\CP_\om$.  In particular,
 $G$ lies in $\CH$ and is representable.
\end{corollary}
\begin{proof}
 Choose a countable family $\CX$ of objects $(X,x)\in\el(F)\sse\el(H)$
 such that the resulting maps $H^X\to F$ are jointly surjective.
 Consider this as a category with only identity morphisms, and enlarge
 it to a countable well-filtered subcategory $\CY\sse\el(H)$.  This is
 in particular filtered, so we can choose a cofinal sequence
 $(X_0,x_0)\to (X_1,x_1)\to (X_2,x_2)\to\dotsb$.  Put
 $P=\bigoplus_nH^{X_n}$ and let $f\:P\to H$ be the map given by $x_n$
 on the $n$th summand.  Let $d\:P\to P$ be the identity minus the
 shift, so $fd=0$, so $f$ factors through a map $f\:\cok(d)\to H$
 (with $\cok(d)=\colim_nH^{X_n}$).  Let $G$ be the image of $f$.  As
 our sequence is cofinal and the category $\CY$ contains $\CX$, we see
 that $F\leq G$.  All that is left is to show that the map
 $\cok(d)\to H$ is injective.  Suppose we have $W\in\CS$ and
 $m\in\cok(d)(W)$ mapping to zero in $H(W)$.  The element $m$ can be
 represented by a map $m\:W\to X_n$ for some $n$, and the image of
 this in $H(W)$ is $m^*(x_n)$, so $m^*(x_n)=0$.  As $\CY$ is
 well-filtered, there is a morphism $p\:(X_n,x_n)\to (Y,y)$ in $\CY$
 with $pm=0$.  As our sequence is cofinal, there is a morphism from
 $(Y,y)$ to $(X_k,x_k)$ for some $k$ such that the composite
 $(X_n,x_n)\to(Y,y)\to(X_k,x_k)$ is just the given map in the
 sequence.  This shows that $m$ maps to zero in the group
 $\cok(d)(W)$, as required.
\end{proof}

\begin{proposition}\label{prop-SES}
 For any $H\in\CH$ there is a short exact sequence
 $0\to Q\xra{d}P\xra{f}H\to 0$ with $P,Q\in\CP$.
\end{proposition}
\begin{proof}
 We can choose an ordinal $\kp$ and a family of objects
 $U_\al\in\CM_p^d$ and elements $u_\al\in H(U_\al)$ for $\al<\kp$ such
 that the resulting map $\bigoplus_{\al<\kp}H^{U_\al}\to H$ is an
 epimorphism.  We aim to construct the following:
 \begin{itemize}
  \item For all $\al<\kp$, objects $Q_\al,P_\al\in\CP_{\om}$, 
   which in turn give objects $Q_{<\al}=\bigoplus_{\bt<\al}Q_\bt$
   and $P_{<\al}=\bigoplus_{\bt<\al}P_\bt$.
  \item Monomorphisms $d_\al\:Q_{<\al}\to P_{<\al}$ for all $\al<\kp$,
   with $d_\al|_{Q_{<\bt}}=d_\bt$ for all $\bt\leq\al$.
  \item Maps $f_\al\:P_{<\al}\to H$ for all $\al<\kp$, such that
   $\ker(f_\al)=\img(d_\al)$, and $\img(f_\al)$ contains
   $\img(u_\bt\:H^{U_\bt}\to H)$ for all $\bt<\al$.
 \end{itemize}
 Suppose we have constructed all this for all $\al<\lm$.  If $\lm$ is
 a limit ordinal, then $P_{<\lm}$ is the colimit of the objects
 $P_{<\al}$ for $\al<\lm$, and similarly for $Q_{<\lm}$, so there is a
 unique possible way to define $d_\lm$ and $f_\lm$.  Suppose instead
 that $\lm=\mu^+$.  As $d_\mu$ is injective, we deduce that the
 functor $\img(f_\mu)\simeq\cok(d_\mu)$ is exact.  As $\img(f_\mu)$
 and $H$ are exact, it follows that the quotient
 $\cok(f_\mu)=H/\img(f_\mu)$ is also exact.  Using
 Corollary~\ref{cor-countable} we can choose a left exact sequence
 $Q_\mu\xra{e}P_\mu\xra{g}\cok(f_\mu)$ with
 $Q_\mu,P_\mu\in\CP_{\om}$, such that $g(P_\mu)$ 
 contains the image of $H^{U_\mu}\to F\to H \to\cok(f_\mu)$.  As
 $H\to\cok(f_\mu)$ is an epimorphism and $P_\mu$ is projective in
 $\CA$ we can choose $g'\:P_\mu\to H$ lifting $g$.  As $ge=0$ we see
 that $\img(g'e)\leq\img(f_\mu)$ and $Q_\mu$ is also projective so we
 can choose $t\:Q_\mu\to P_{<\mu}$ with $g'e=f_\mu t$.  Now note that
 $Q_{<\lm}=Q_{<\mu}\oplus Q_\mu$ and $P_{<\lm}=P_{<\mu}\oplus P_\mu$
 so we can define maps $Q_{<\lm}\xra{d_\lm}P_{<\lm}\xra{f_\lm}H$ by 
 $d_\lm(x,y)=(d_\mu(x)-t(y),e_\mu(y))$ and
 $f_\lm(x,y)=f_\mu(x)+g'(y)$.  By chasing the diagram 
 \begin{center}
  \begin{tikzcd}[ampersand replacement=\&]
   Q_{<\mu}
    \arrow[d,rightarrowtail,"d_\mu"'] 
    \arrow[r,rightarrowtail,"{\bsm 1\\ 0\esm}"] \&
   Q_{<\mu}\oplus Q_\mu
    \arrow[d,"{\bsm d_\mu & -t\\ 0&e_\mu\esm}"]
    \arrow[r,twoheadrightarrow,"{\bsm 0&1\esm}"] \&
   Q_\mu 
    \arrow[d,rightarrowtail,"e"] \\
   P_{<\mu}
    \arrow[d,twoheadrightarrow,"f_\mu"'] 
    \arrow[r,rightarrowtail,"{\bsm 1\\ 0\esm}"] \&
   P_{<\mu}\oplus P_\mu
    \arrow[d,"{\bsm f_\mu & g'\esm}"]
    \arrow[r,twoheadrightarrow,"{\bsm 0&1\esm}"] \&
   P_\mu 
    \arrow[d,"g"] \\
   \img(f_\mu)
    \arrow[r,rightarrowtail] \&
   H 
    \arrow[r,twoheadrightarrow] \&
   \cok(f_\mu)
  \end{tikzcd}
 \end{center}
 we see that the requirements for stage $\lm$ are now satisfied.  By
 transfinite recursion, we can now define everything up to stage
 $\kp$.  We then have a short exact sequence
 $Q_{<\kp}\to P_{<\kp}\to H$ of the required type.
\end{proof}

\begin{proof}[Proof of Theorem~\ref{thm-brown}]
 Combine Proposition~\ref{prop-SES} and Lemma~\ref{lem-cok-rep}.
\end{proof}

\bibliographystyle{alpha}

\end{document}